\documentclass{amsart}

\usepackage{amsmath,amsthm,amsfonts}
\usepackage{mathtools}
\usepackage{mathrsfs}

\usepackage{tikz}
\usepackage{hyperref}

\usepackage{comment}

\usepackage[english]{babel}

\theoremstyle{plain}

\newtheorem{thm}{Theorem}

\newtheorem{prop}[thm]{{\bf Proposition}}
\newtheorem{lem}[thm]{{\bf Lemma}}

\newcounter{hyp_counter}
\theoremstyle{definition}
\newtheorem{claim}{Claim}
\newtheorem{defn}[thm]{Definition}

\theoremstyle{remark}
\newtheorem{rem}{Remark}
\newtheorem{question}{Question}

\newcommand{\Id}{\operatorname{Id}}

\newcommand{\Diff}{\operatorname{Diff}}
\newcommand{\GL}{\operatorname{GL}}

\newcommand{\Per}{\operatorname{Per}}

\newcommand{\R}{\mathbb{R}}

\newcommand{\SL}{\operatorname{SL}}

\newcommand{\vol}{\operatorname{vol}}

\newcommand{\Z}{\mathbb{Z}}

\newcommand{\wt}[1]{\widetilde{#1}}

\newcommand{\abs}[1]{\left| #1\right|}

\newcommand{\mc}[1]{\mathcal{#1}}

\newcommand{\pez}[1]{\left( #1\right)}

\makeatletter
\def\blfootnote{\xdef\@thefnmark{}\@footnotetext}
\makeatother

\title[Rigidity of Anosov automorphisms with Jordan Blocks]{Periodic data rigidity of Anosov automorphisms with Jordan blocks}
\author{Jonathan DeWitt}
\address{Department of Mathematics, The University of Maryland, Maryland, MD 20742, USA}
\email{dewitt@umd.edu}
\date{\today}

\begin{document}

\begin{abstract}
Anosov automorphisms with Jordan blocks are not periodic data rigid. We introduce a refinement of the periodic data and show that this refined periodic data characterizes $C^{1+}$ conjugacy for Anosov automorphisms on $\mathbb{T}^4$ with a Jordan block.
\end{abstract}

\maketitle

\section{Introduction}

An Anosov diffeomorphism of a Riemannian manifold $M$ is a diffeomorphism $F\colon M\to M$ such that $TM$ continuously splits into two $DF$-invariant bundles $E^u$ and $E^s$ such that vectors in $E^u$ are uniformly expanded by $DF$ and vectors in $E^s$ are uniformly contracted by $DF$. If $L\in \SL(n,\Z)$ does not have any eigenvalues of unit modulus, then the map induced by $L$ on $\mathbb{T}^n$ is an Anosov diffeomorphism, which is called an \emph{Anosov automorphism} due to its algebraic construction. Anosov diffeomorphisms of tori exhibit strong topological rigidity. If $F\colon \mathbb{T}^n\to \mathbb{T}^n$ is Anosov and in the homotopy class of $L\in \SL(n,\Z)$, then by work of Franks \cite{franks1969anosov} and Manning \cite{manning1974there}, there exists a homeomorphism $h$ such that $h^{-1} F h=L$. The map $h$ is called a \emph{conjugacy} and is H\"older continuous. In this paper, we study ridigity of  Anosov automorphisms defined by a matrix $L\in \SL(4,\Z)$ that have a Jordan block.

This paper is the first to show rigidity of an Anosov diffeomorphism with a Jordan block. There are two main contributions of this paper. The first is the identification of the correct refinement of the periodic data for this setting and showing that this data exists and is well-defined. The second main contribution is a new approach to studying the regularity of conjugacies that does not rely on either conformality or an abundance of dynamically invariant foliations. In fact, in this paper, we contend with Anosov automorphisms that restricted to their unstable manifold preserve only a single $1$-dimensional foliation and whose differential has polynomial growth of conformal distortion. The development of techniques in this setting is useful in other contexts because many systems that we expect to be rigid do not admit many invariant foliations.

In this paper, we study a rigidity property of conjugacies between Anosov diffeomorphisms. Specifically, we will investigate conditions that imply that a conjugacy $h$ is $C^1$. One well known obstruction to the existence of a $C^1$ conjugacy is the periodic data, which we now describe. Suppose that $f$ and $g$ are two diffeomorphisms that are $C^1$ conjugate by a conjugacy $h$. If $p$ is a periodic point of period $n$, then $h(p)$ is a periodic point of $g$ of period $n$. By the chain rule, we must have that
\[
D_pf^n=D_{h(p)}hD_{h(p)}g^n(D_{h(p)}h)^{-1}.
\]
Thus the return maps of $Df$ and $Dg$ at corresponding periodic points are conjugate. Given two diffeomorphisms $f$ and $g$ with a conjugacy $h$ between them, we say that $f$ and $g$ have the same \emph{periodic data with respect to $h$} if for each periodic point $p$ of period $n$, we have that $D_pf^n$ and $D_{h(p)}g^n$ are conjugate as linear maps. In some situations, if $f$ and $g$ have the same periodic data then $h$ is $C^1$.

We say that a diffeomorphism is $C^{k+}$ when it is $C^k$ and its $k$th derivative is $\alpha$-H\"older continuous for some $\alpha>0$. We write $\Diff^{k+}(M)$ for the group of diffeomorphisms of $M$ that are $C^{k+}$. This leads us to the definition of rigidity we are interested in here.
\begin{defn}
We say that an Anosov diffeomorphism $f\colon M\to M$ is \emph{periodic data rigid} if for any Anosov diffeomorphism $g\in \Diff^{2+}(M)$ in the homotopy class of $f$, if $f$ and $g$ have the same periodic data with respect to a conjugacy $h$, then $h$ is $C^{1+}$.
\end{defn}
\noindent The main examples of diffeomorphisms $f$ exhibiting this type of rigidity are Anosov automorphisms. 

Periodic data rigidity is well studied. Early work was done by De la Llave, Marco, and Moriyon. See, for example, \cite{de1992smooth}, \cite{marco1987invariantsI}, and \cite{de1987invariants}. More recently, a paper by Gogolev, Kalinin and Sadovskaya \cite{gogolev2020local}, showed local periodic data rigidity of an Anosov automorphism of $\mathbb{T}^n$, under the assumption that no three eigenvalues of $L$ have the same modulus and that $L$ and $L^4$ are both irreducible. For earlier periodic data rigidity results implied by the work of \cite{gogolev2020local}, see for example \cite{gogolev2008smooth} and \cite{gogolev2011local}. Recently, Saghin and Yang obtained some additional results on the torus \cite{saghin2019lyapunov}. 
The previously mentioned results apply to Anosov automorphisms of tori. The author recently also obtained periodic data rigidity results for Anosov automorphisms of nilmanifolds \cite{dewitt2021local}. Recently Gogolev and Rodriguez Hertz \cite{gogolev2022smooth}, have proved additional periodic data rigidity results relying on a novel condition called \emph{very non-algebraicity} using a technique of matching functions that they introduced. They also extended these results to codimension one Anosov diffeomorphisms in \cite{gogolev2021smooth}.

In this paper we consider Anosov automorphisms of $\mathbb{T}^4$ with a Jordan block, such as the following:
\begin{equation}
A=\begin{bmatrix}
2 & 1 & 1 & 0\\
1 & 1 & 0 & 1\\
0 & 0 & 2 & 1\\
0 & 0 & 1 & 1
\end{bmatrix}.
\end{equation}
Such automorphisms with Jordan blocks are known to not be periodic data rigid due to examples of de la Llave in \cite[Sec.~6]{delallave2002rigidity}, and do not seem to have been studied after the construction of those counterexamples. In fact, any Anosov automorphism of $\mathbb{T}^4$ that has a Jordan block in its periodic data is not periodic data rigid. In this paper, we nonetheless recover a periodic data rigidity result for such automorphisms by introducing a refinement of the periodic data, which we call the Jordan periodic data. 

We now describe what the Jordan periodic data is. Fix some $\lambda>0$. If we have a cocycle taking values in the set of matrices
\[
U_{\lambda}=\left\{\begin{bmatrix}
\lambda & a\\
0 & \lambda
\end{bmatrix}\mid a\in \R\right\},
\]
then knowing that two such cocycles have the same periodic data may not be enough to determine that they are cohomologous as $\GL(2,\R)$ cocycles. In fact, if we have two continuous cocycles $\mc{A}$ and $\mc{B}$ over a map $\sigma\colon \Sigma\to \Sigma$ taking values in  $U_{\lambda}$, and all of the periodic data of $\mc{A}$ and $\mc{B}$ has a Jordan block, then their periodic data is identical because all such Jordan blocks are conjugate. 
If we write these cocycles as 
\[
\begin{bmatrix}
\lambda & \alpha_A\\
0 &\lambda
\end{bmatrix}\text{ and }
\begin{bmatrix}
\lambda & \alpha_B\\
0 & \lambda
\end{bmatrix},
\]
then these two cocycles are continuously cohomologous as $\GL(2,\R)$ cocycles precisely when there exists a constant $C$ and a continuous map $\psi\colon \Sigma\to \R$ such that 
\[
\alpha_A=C\alpha_B+\psi\circ \sigma-\psi.
\]
We write $[\alpha_A]$ for the class of all functions such as $\alpha_B$ that $\alpha_A$ is ``projectively" cohomologous to. 
This ``projective cohomology class" is itself determined by the projective class of $[\alpha_A]$'s periodic data. 
Hence there is no loss of thinking of $[\alpha_A]$ as periodic data.  We refer to $[\alpha_A]$ as the \emph{Jordan periodic data} associated to $\mc{A}$.

Not all cocycles are immediately recognizable as taking values in $U_{\lambda}$. We say that a cocycle is \emph{$U_{\lambda}$-framed} if it may be continuously conjugated into $U_{\lambda}$. If a cocycle is $U_{\lambda}$-framed and has a Jordan block, then its Jordan periodic data may be defined, and indeed is well defined independent of the conjugacy into $U_{\lambda}$.

Our main result has two parts. The first part shows that if an Anosov diffeomorphism has the same periodic data as an Anosov automorphism with a Jordan block, then it has well defined Jordan periodic data.

\begin{prop}\label{prop:framings_exist}
Let $L$ is an Anosov automorphism of $\mathbb{T}^4$ with a Jordan block and eigenvalues $\lambda, \lambda^{-1}$. Suppose that $F$ is a $C^2$ Anosov diffeomorphism with the same periodic data as $L$. Then the stable and unstable bundles of $F$ are trivial and $U_{\lambda}$ and $U_{\lambda^{-1}}$-framed, respectively.
\end{prop}

In the case that the stable and unstable bundles admit such framings, one obtains $[\alpha_u^F]$ and $[\alpha_s^F]$ as before; we refer to these as the unstable and stable Jordan periodic data, respectively. We refer to the pair $[\alpha_u^F]$ and $[\alpha_s^F]$ as the \emph{full} Jordan periodic data.

The second part of the main result shows that if an Anosov diffeomorphism has the same Jordan periodic data as an Anosov automorphism, then the two are $C^{1+}$-conjugate.

\begin{prop}\label{prop:framed_implies_smooth}
Suppose that $L$ is an Anosov automorphism of $\mathbb{T}^4$ with a Jordan block and that $F$ is a $C^{2+}$ Anosov diffeomorphism such that
\begin{enumerate}
\item $F$ has the same periodic data as $L$ with respect to a conjugacy $h$;
\item The stable and unstable bundles of $F$ are trivial and admit $U_{\lambda^{-1}}$ and $U_{\lambda}$-framings, respectively;
\item
$F$ has the same full Jordan periodic data as $L$;
\end{enumerate}
Then $h$ is a $C^{1+}$ conjugacy between $F$ and $L$.
\end{prop}
These immediately combine to prove the following theorem:

\begin{thm}\label{thm:main_thm}
If $L\in \SL(4,\Z)$ defines an Anosov automorphism of $\mathbb{T}^4$ with a Jordan block and $F$ is any $C^{2+}$ Anosov diffeomorphism of $\mathbb{T}^4$ in the homotopy class of $L$ with the same periodic data as $L$, then the two have well defined Jordan periodic data and are $C^{1+}$ conjugate if and only if their full Jordan periodic data coincides.
\end{thm}

In fact, by using a recent result of Kalinin, Sadovskaya, and Wang \cite[Thm.~1.3]{kalinin2022local}, one can deduce local $C^{\infty}$ regularity of the conjugacy if we assume that $F$ is $C^{\infty}$ and a sufficiently high regularity perturbation of $L$.

The Anosov automorphisms to which Theorem \ref{thm:main_thm} applies may be described quite succinctly. The proof of the following appears in the appendix.

\begin{prop}\label{prop:integer_conjugacy_classes}
Suppose that $L\in \SL(4,\Z)$ is a matrix with a Jordan block and no eigenvalues of modulus $1$. Then $L$ is conjugate to a block matrix of the form
\[
\begin{bmatrix}
A & C\\
0 & B
\end{bmatrix},
\]
where $A,B\in \GL(2,\Z)$, $C\neq 0$, and the characteristic polynomials of $A$ and $B$ are equal and have two distinct real roots.
\end{prop}

There are some other results that one can potentially prove by combining the techniques in this paper with other work. Here are a few:
\begin{enumerate}
\item
One could prove local results for Anosov automorphisms with more than one unstable Lyapunov exponent by combining the techniques here with those in \cite{gogolev2011local} and \cite{dewitt2021local}.
\item
It seems quite likely that the techniques introduced in this paper allow one to prove analogous theorems for Anosov automorphisms with large dimensional Jordan blocks. 

\end{enumerate}

The author does not know the answer to the following question, which seems like it requires some additional techniques.

\begin{question}
Suppose that $f$ and $g$ are Anosov diffeomorphisms with the same periodic data as a linear Anosov diffeomorphism of $\mathbb{T}^4$ that has a Jordan block. If the full Jordan data of $f$ and $g$ coincide, then are $f$ and $g$ $C^1$ conjugate?
\end{question}

\noindent\textbf{Acknowledgements:} The author thanks Daniel Mitsutani for useful comments on this paper. The author is also grateful to Aaron Brown, Andrey Gogolev, Boris Hasselblatt, and Amie Wilkinson for helpful discussions.

\subsection{Outline of proof}
The proof follows the following outline.
\begin{enumerate}
\item
We introduce the ``Jordan" periodic data.
\item
We do some preliminary work to produce a $DF$-invariant flag $\mc{E}$ on which the Jordan periodic data is well defined.
\item
Show the existence of ``slow" foliations within each unstable leaf tangent to $\mc{E}$ and their unique integrability. 
\item
Show that the conjugacy $h$ intertwines these slow foliations and is uniformly $C^{1+}$ along them.
\item
Show that the ``slow" foliation is $C^{1+}$ and that its holonomies are uniformly regular over long distances.
\item
Using the de la Llave argument \cite{delallave2002rigidity}, we construct a $C^1$ model of the conjugacy, $h_0$, such that $F^{n}\circ h_0\circ L^{-n}\to h$ uniformly. By studying the derivatives $F^nh_0L^{-n}$, we obtain that $h$ is Lipschitz.
\item
Once we know that $h$ is suitably Lipschitz, we can differentiate it and use cocycle rigidity to conclude that its derivative is continuous and hence $h$ is differentiable.
\end{enumerate}

\section{Preliminaries}\label{sec:anosov_rigidity}

For definitions concerning foliations, we refer the reader to \cite{pugh1997holder}, which contains a thorough discussion of the topic.  We use the same terminology as that paper, which is standard. We now set some notation that will be of particular use in our argument. For a foliation $\mc{F}$, we write $\mc{F}(x)$ for the leaf containing the point $x$. As an extension of this notation, if $S$ is a set, then we write $\mc{F}(S)$ for $\bigcup_{x\in S} \mc{F}(x)$. If $\mc{F}$ is a foliation with $C^1$ leaves that foliates a Riemannian manifold $M$, if $x\in \mc{F}(y)$ then we denote by $d_{\mc{F}}(x,y)$ the distance between $x$ and $y$ as measured along the immersed submanifold $\mc{F}$, where $\mc{F}$ is endowed with the pullback Riemannian metric obtained from its inclusion into $M$.  If $h\colon M\to M$ is a map, then we say that $h$ intertwines the foliations $\mc{F}$ and $\mc{G}$ if $h(\mc{F}(x))=\mc{G}(h(x))$ for all points $x\in M$.

We say that a map $f$ is $C^{1+}$ if there exists $0<\alpha<1$ such that the derivative of $f$ is $C^{\alpha}$. When we say that a map of a noncompact space is uniformly $C^{\alpha}$, we mean that there exist $\alpha,\delta>0$ such that restricted to balls of radius $\delta$, the map is $C^{\alpha}$ H\"older with uniform constant.

In the following argument, we will often work with uniform transversals to a $C^1$ foliation $\mc{F}\subset M$. What we mean by this is that we are considering transversals $\gamma$ such that $\gamma$ is a uniformly $C^1$ map from its domain, which is an interval, to $M$ and $\dot{\gamma}(x)$ is uniformly transverse to $T_{\gamma(x)}\mc{F}$.

If $W$ is a transversal to a foliation $\mc{F}$ and each leaf of $\mc{F}$ intersects $W$ at most once, then we write $\Pi_{W}$ for the projection onto $W$ along the leaves of $\mc{F}$.  Specifically, if $z$ is a point such that $z\in \mc{F}(x)$ for some $x\in W$, then $\Pi_W(z)=x$. We say that transversals $T$ and $W$ are holonomy related if $W\subset \mc{F}(T)$ and vice versa.

For an Anosov diffeomorphism $F$, we denote by $\mc{W}^{u,F}$ the unstable foliation of $F$. This foliation has uniformly $C^{2+}$ leaves when $F$ is $C^{2+}$. We denote by $E^{u,F}$ the unstable bundle of $F$; similarly we define the stable foliation $\mc{W}^{s,F}$ and the stable bundle $E^{s,F}$.

If $E$ is a $n$-dimensional vector bundle over a manifold $M$, then by a \emph{framing} of $E$ we mean a choice of basis $[e_1,\ldots,e_n]$ of $E_x$ for each $x$ that varies continuously with $x$. For an Anosov automorphism $L$ on $\mathbb{T}^4$ with a Jordan block, we have a translation invariant framing $[e_1^L,e_2^L]$ that presents the differential of $DL\vert_{E^{u,L}}$ as a Jordan block:
\[
\begin{bmatrix}
\lambda & 1\\
0 & \lambda
\end{bmatrix}.
\]

\section{Jordan Periodic Data}

Suppose that $A\colon \Sigma\to \GL(2,\R)$ is a continuous function defining a cocycle over transitive hyperbolic system $\sigma\colon \Sigma\to \Sigma$, such as a transitive Anosov diffeomorphism. Suppose that $A$ is continuously conjugate to a cocycle taking values in the group $U_{\lambda}\subset \GL(2,\R)$ given by
\[
U_{\lambda}=\left\{\begin{bmatrix}
\lambda & a\\
0 & \lambda
\end{bmatrix}\mid a\in \R\right\}.
\]
We say that such a cocycle is reducible to a cocycle taking values in $U_{\lambda}$, or is $U_{\lambda}$-framed.
  By assumption, the cocycles of interest in this paper have non-trivial Jordan blocks in their periodic data. 
If a cocycle has a Jordan block in its periodic data for every periodic point, then we say that this cocycle has \emph{Jordan-full} periodic data.
 We now consider the cohomology of such cocycles that are reducible to $U_{\lambda}$ and have Jordan-full periodic data. We may write such a cocycle as 
\begin{equation}
\begin{bmatrix}
\lambda & \alpha_A\\
0 & \lambda
\end{bmatrix},
\end{equation}
for $\alpha_A\colon \Sigma\to \R$ a function whose regularity is the same as the regularity of $A$. The function $\alpha_A$ characterizes $A$ up to continuous $\GL(2,\R)$-conjugacy: two $U_{\lambda}$-valued coycles $A$ and $B$ are continuously conjugate if and only if $\alpha_A$ and $\alpha_B$ lie in the same ``projective" cohomology class. The following follows from a much more detailed result of Sadovskaya \cite[Prop. 5.1]{sadovskaya2013cohomology}, which gives a detailed description of $\GL(2,\R)$ cocycles.

\begin{lem}\label{lem:jordan_data_cohomology}
Suppose that $\Sigma$ is a transitive hyperbolic system and that $A,B\colon \Sigma\to \GL(2,\R)$ are two H\"older continuous functions defining cocycles reducible to $U_{\lambda}$ that are Jordan-full. Then $A$ and $B$ are H\"older conjugate if and only if there exists a constant $C\neq 0$ and a H\"older continuous function $\phi\colon \Sigma\to \R$ such that 
\[
\alpha_A=C\alpha_B+\phi\circ \sigma-\phi.
\]
In fact, if $A$ and $B$ take values in $U_{\lambda}$, then any conjugacy between $A$ and $B$ takes values in upper triangular matrices.
\end{lem}

By the usual abelian Livsic theorem, $\alpha_A=C\alpha_B+\phi\circ \sigma-\phi$ if and only if there exists $C$ such that for each periodic point $p$ of period $k$, $\sum_{i=1}^k \alpha_A(f^ip)=C\sum_{i=1}^k \alpha_B(f^ip)$. If there is such a $C$, we say that the periodic data of $A$ and $B$ is in the same \emph{projective class}. Given this preliminary the following definition is well-defined. 

\begin{defn}
Suppose that $A\colon \Sigma\to \GL(2,\R)$ is a function defining a cocycle over a transitive hyperbolic system that is reducible to $U_{\lambda}$. Let $\Per(\Sigma)$ be the set of periodic points of $\Sigma$. As before, associated to $A$ is the function $\alpha_A\colon \Sigma\to \R$. We define the \emph{Jordan} periodic data of $A$ to be the projective class of the function $\Per(\Sigma)\to \R$ that sends a periodic point $p$ of period $k$ to the sum of $\alpha_A$ along its orbit:
\[
p\mapsto \sum_{i=1}^k \alpha_A(\sigma^i(p)).
\]
\end{defn}
Lemma \ref{lem:jordan_data_cohomology} shows that the Jordan periodic data characterizes up to conjugacy the cocycles for which it is defined.

The discussion above defines the Jordan periodic data for cocycles over a single fixed system. If we have two cocycles over different conjugate systems, then we may pull back the cocycle by the conjugacy. We then say that two cocycles have the same Jordan periodic data with respect to a conjugacy $h$ if the pulled back cocycle has the same Jordan periodic data as the original cocycle.

If $F$ is an Anosov diffeomorphism such that $E^{u,F}$ and $E^{s,F}$ are trivial bundles and admit $U_{\lambda}$ and $U_{\lambda^{-1}}$-framings, then we write $[\alpha_u^F]$ for the Jordan data of $DF\vert_{E^{u,F}}$ and $[\alpha_s^F]$ for the Jordan data of the stable bundle. We call the pair the \emph{full} Jordan data of $F$.
\section{Preliminary Reductions}

In this section we give some preliminary reductions that bring us closer to producing the $U_{\lambda}$-framings we need to define the periodic data.

\begin{lem}\label{lem:flag_exists}
Suppose $L$ is as in Proposition \ref{prop:integer_conjugacy_classes} and that $F$ is a $C^2$ Anosov diffeomorphism with the same periodic data as $L$. Then there exists a nontrivial H\"older continuous $DF$-invariant flag
\[
0\subset \mc{E}\subset E^{u,F}
\]
and H\"older continuous Riemannian metrics on $\mc{E}$ and $E^{u,F}/\mc{E}$ such that with respect to these metrics the differential of $F$ has norm exactly $\lambda$. The same holds for $E^{s,F}$.
\end{lem}

\begin{proof}
Suppose that $F$ is an Anosov diffeomorphism with the same periodic data as $L$. 
Because $F$ has the same periodic data as $L$, the top and bottom Lyapunov exponents of the measures supported on periodic orbits are equal. 
Thus we may apply the continuous amenable reduction \cite[Thm 3.9]{kalinin2013cocycles} to conclude that there exists a H\"older continuous $DF$ invariant flag
\[
0\subset \mc{E}\subset E^{u,F}.
\]
We now appeal to our knowledge of the periodic data for the claim about the metric. Note that because $\mc{E}$ is one dimensional that the map $N\colon \mathbb{T}^4 \to \R^{\times}$ given by $x\mapsto \|D_xF\vert_{\mc{E}}\|$ defines a cocycle over $F$. Further, note that the assumption on periodic data implies that all the periodic data of $N$ is a power of $\lambda$. Thus $N$ by the abelian Livsic theorem \cite[Thm. 19.2.1]{katok1997introduction}, $N$ is H\"older cohomologous to the constant cocycle $\lambda$, i.e. $N=\lambda(\phi\circ F)\phi^{-1}$ for some H\"older $\phi\colon \mathbb{T}^4\to \R^{\times}$. If we replace the reference metric $\|\cdot \|$ on $\mc{E}$ with the metric $\phi\|\cdot\|$, then for this new metric we have that the norm of $DF\vert_{\mc{E}}$ is exactly $\lambda$ at every point. 

Precisely the same considerations show that $E^{u,F}/\mc{E}$ admits a norm with the same property.
\end{proof}

\begin{rem}
The reason that the above lemma is phrased in terms of a norm instead of in terms of a cocycle conjugacy is that there is nothing telling us that the bundle $\mc{E}$ is orientable.
\end{rem}

\section{The ``slow" foliation}\label{subsec:slow_foliation}
Note that if we identify a $\mc{W}^{u,L}$ leaf with $\R^2$ by use of the framing $[e_1^L,e_2^L]$, that $\mc{W}^{u,L}$ has an invariant foliation tangent to the $e_1^L$ line field. We call this foliation $\mc{S}^L$. We think of this foliation as ``slow" because vectors tangent to it grow at only rate $O(\lambda^n)$ when we iterate $L$. The following proposition constructs an analogous ``slow" foliation $\mc{S}^F$ that subfoliates the unstable foliation $\mc{W}^{u,F}$. 

In this proof we will use the notion of a quasi-isometry. Consider a surjection of metric spaces $f\colon (M_1,d_1)\to (M_2,d_2)$. Then $f$ is an $(A,B)$-quasi-isometry if there exist constants $A\ge 1$ and $B\ge 0$ such that for any $x,y\in M_1$,
\[
\frac{1}{A}d_1(x,y)-B\le d_2(f(x),f(y))\le Ad_1(x,y)+B.
\]
Informally, one thinks of quasi-isometries as being ``Lipschitz at large scale."

\begin{prop}\label{prop:h_intertwines_slow_leaves}
Suppose that $F$ is a $C^{2+}$ Anosov diffeomorphism that has the same periodic data as $L\in \SL(4,\Z)$, which is hyperbolic and has a Jordan block, and that $h$ is a conjugacy between $F$ and $L$.
Restricted to a $\mc{W}^{u,F}$ leaf, the bundle spanned by $e_1^F$ is uniquely integrable and tangent to a foliation $\mc{S}^F$, which has uniformly $C^{1+}$ leaves. Further, $h$ intertwines $\mc{S}^F$ with the foliation $\mc{S}^L$.
\end{prop}

\begin{proof}
By \cite[Cor. 2.7]{dewitt2021local}, there exist uniform constants $A,B$ such that for all $x\in \mathbb{T}^4$, $h\vert_{\mc{W}^{u,L}(x)}\colon \mc{W}^{u,L}(x)\to \mc{W}^{u,F}(h(x))$ is an $(A,B)$-quasi-isometry.

Note that
\[
\begin{bmatrix}
\lambda & 1\\
0 & \lambda
\end{bmatrix}^n=
\lambda^n \begin{bmatrix}
1 & n/\lambda \\
0 & 1
\end{bmatrix}.
\]
This implies that if $x\in \mc{W}^{u,L}(y)$ and $x\notin \mc{S}^L(y)$, then there exists $C>0$ such that
\[
d_{\mc{W}^{u,L}}(L^n(x),L^n(y))\ge C\lambda^n n.
\]
As $h$ is a quasi-isometry, this implies that
\begin{equation}\label{eqn:nlambda_growth}
d_{\mc{W}^{u,F}}(F^n(h(x)),F^n(h(y)))=d_{\mc{W}^{u,F}}(h(L^n(x)),L^n(y))\ge A^{-1}C\lambda^n n-B.
\end{equation}
But if $q\in \mc{W}^{u,F}(r)$ and $q$ and $r$ are two points connected by a curve tangent to $\mc{E}$ contained within $\mc{W}^{u,F}(r)$, then
\begin{equation}\label{eqn:lambdan_growth}
d_{\mc{W}^{u,F}}(F^n(q),F^n(r))\le \lambda^n
\end{equation}
because $\|DF\vert_{\mc{E}}\|=\lambda$.
Thus if $q$ and $r$ are connected by such a curve, then $h^{-1}(q)$ and $h^{-1}(r)$ must lie in the same $\mc{S}^L$ leaf inside of $\mc{W}^{u,L}$ because otherwise equation \eqref{eqn:lambdan_growth} would contradict equation \eqref{eqn:nlambda_growth}. This implies that the curves inside an unstable leaf tangent to $\mc{E}$ uniquely integrate to a foliation. Further, this foliation has uniformly $C^{1+}$ leaves because the $\mc{E}$ distribution is uniformly H\"older by Proposition \ref{lem:flag_exists}.
\end{proof}

\section{Differentiability along the slow foliation and the Jordan periodic data}

The approach we follow for showing that $h$ is differentiable along $\mc{S}^L$ is not new and follows exactly the same argument as in \cite[Sec. 2.5]{dewitt2021local}. Consequently, we will just give an outline of the idea.
To establish differentiability of $h$ along $\mc{S}^L$,
we will use the following result, which is a $C^{1+}$ version of Saghin and Yang \cite[Thm. G]{saghin2019lyapunov}. All the definitions used in this section and a thorough discussion may be found in \cite[Sec. 2.5]{dewitt2021local}.

\begin{defn}
Let $\mc{F}$ be an expanding foliation for a $C^{1+}$ diffeomorphism $f$. An $f$-invariant measure $\mu$ is a \emph{Gibbs expanding state} along $\mc{F}$ if for any foliation chart of $\mc{F}$, the disintegration of $\mu$ along the plaques of the chart is equivalent to the Lebesgue measure on the plaque for $\mu$-almost every plaque.
\end{defn}

Using this definition, we may state the following.

\begin{lem}\cite[Lem. 13]{dewitt2021local}.
Let $M$ be a smooth closed manifold, and let $f,g\in \Diff^{1+}(M)$. Let $\mc{F}$ be a one-dimensional expanding foliation for $f$, and let $\mc{G}$ be an expanding foliation for $g$ such that $\mc{F}$ and $\mc{G}$ have uniformly $C^{1+}$ leaves. Let $\mu$ be a Gibbs expanding state of $f$ along $\mc{F}$. Suppose that $f$ and $g$ are topologically conjugate by a homeomorphism $h$ and that $h$ intertwines $\mc{F}$ and $\mc{G}$. Then the following two conditions are equivalent:
\begin{enumerate}
\item
$\nu\coloneqq h_*(\mu)$ is a Gibbs expanding state of $g$ along the foliation $\mc{G}$.
\item
$h$ restricted to each $\mc{F}$ leaf within the support of $\mu$ is uniformly $C^{1+}$.

\end{enumerate}
\end{lem}

The main thing that needs to be verified in order to apply the previous lemma is that the $\mc{S}^F$ foliation is absolutely continuous. This is immediate from the following lemma of Ledrappier because $\|DF\vert_{\mc{E}}\|=\lambda$ is constant.

\begin{lem}\cite{ledrappier1985metric}
Let $f$ be a $C^{1+}$ diffeomorphism and let $\mu$ be an $f$-invariant measure. Suppose that $\mc{F}$ is an expanding foliation for $f$. Suppose that $\xi$ is an increasing measurable partition subordinate to $\mc{F}$ and $\mu$. Then the conditional measures of $\mu$ are absolutely continuous on the leaves of $\mc{F}$ if and only if
\[
H_{\mu}(f^{-1}\xi\mid \xi)=\int \log\|Df\mid_{T\mc{F}}\|\,d\mu,
\]
where $H_{\mu}(f^{-1}\xi\mid \xi)$ is the conditional entropy of $f^{-1}\xi$ given $\xi$.
\end{lem}

The previous two lemmas then combine to show the following.

\begin{prop}\label{prop:smooth_on_slow_leaves}
The conjugacy $h$ in Proposition \ref{prop:h_intertwines_slow_leaves}
 is uniformly $C^{1+}$ on leaves of the $\mc{S}^L$ foliation.
\end{prop}

Using this differentiability we can show that:

\begin{prop}\label{prop:E_is_trivial}
The bundle $\mc{E}$ in Lemma \ref{lem:flag_exists} is trivial.
\end{prop}
\begin{proof}
Recall the linear frame $[e_1^L,e_2^L]$ of $E^{u,L}$. Because $h$ intertwines $\mc{S}^L$ and $\mc{S}^F$, and is $C^{1+}$ along these foliations, we see that $Dh(e_1^L)$ is a continuous framing of $\mc{E}$; hence $\mc{E}$ is trivial.
\end{proof}

We can now prove Proposition \ref{prop:framings_exist}, which says that if $F$ has the same periodic data as $L$, which has Jordan blocks, then the stable and unstable bundles of $F$ are $U_{\lambda}$ and $U_{\lambda^{-1}}$-framed, respectively.

\begin{proof}[Proof of Proposition \ref{prop:framings_exist}.]
By Proposition \ref{prop:E_is_trivial} the bundle $\mc{E}$ from
 Lemma \ref{lem:flag_exists} has a framing $\wt{e}_1^F$.
 The periodic data for $DF\vert_{\mc{E}}$ is exactly multiplication by $\lambda$. Hence by application of Livsic we can rescale the $\wt{e}_1^F$ framing to find a new framing $e_1^F$ with the property that
\[
D_xFe_1^F(x)=\lambda e_1^F(F(x)).
\]

We now show that $E^u$ and $E^u/\mc{E}$ are trivial. By Proposition \ref{prop:Anosov_unstable_is_orientable}, the bundle $E^u$ is orientable, hence $w_1(E^u)=0$. But $E^u=E^u/\mc{E}\oplus \mc{E}$, so $w_1(E^u)=w_1(E^u/\mc{E})\oplus \mc{E})$ by additivity of the first Stiefel-Whitney class. Thus we find that $w_1(E^u/\mc{E})=0$, hence this bundle is trivial. This also implies that $E^u$ is trivial as it is the sum of trivial bundles.

Thus we see that $E^{u,F}/\mc{E}$ also admits a non-vanishing section $\overline{e}_2^F$ because it is trivial. By the same Livisic argument as before, we can rescale $\overline{e}_2^F$ so that $DF$ acts on this section as multiplication by $\lambda$. Let $e_2^F$ be a H\"older continuous section of $E^u$ projecting to $\overline{e}_2^F$ along $\mc{E}$. Then with respect to the framing $[e_1^F,e_2^F]$, $DF\vert_{E^u}$ has the form described above. The same considerations apply to the stable bundle.
\end{proof}

\begin{defn}\label{defn:U_lambda_framed}
In the sequel, we will say that an Anosov diffeomorphism is $U_{\lambda}$-framed if the following hold:
\begin{enumerate}
\item
$F$ is in the same homotopy class as an Anosov automorphism $L\in \SL(4,\Z)$ with a Jordan block;
\item
The stable and unstable bundles of $F$ are trivial;
\item
$DF\vert_{E^{u,F}}$ and $DF\vert_{E^{s,F}}$ admit a $U_{\lambda}$- and $U_{\lambda^{-1}}$-framings, where $\lambda>1$ is an eigenvalue of $L$. We refer to the unstable framing as $[e_1^F,e_2^F]$.
\end{enumerate}
\end{defn}

\section{Regularity of the ``slow" foliation}

In order to upgrade the regularity of the ``slow" foliation, we will use the normal forms developed by Kalinin and Sadovskaya.
The following Theorem is a combination of Theorem 4.6, Corollary 4.8, and Remark 4.2 in \cite{kalinin2020nonstationary}.

\begin{thm}\label{thm:normal_forms}
(Non-stationary Linearization)
Let $f$ be a $C^r$, $r\in (1,\infty]$ diffeomorphism of a smooth manifold $X$, and let $\mc{W}$ be an $f$-invariant topological foliation of $X$ with uniformly $C^r$ leaves. Let $\chi>0$ and 
\[
0<\epsilon< \chi/(d+4).
\]
Suppose that the linear extension $Df\vert_{T\mc{W}}$ satisfies:
\begin{equation}\label{eqn:normal_forms_condition}
e^{\chi-\epsilon}\|v\|\le \|D_xf(v)\|\le e^{\chi+\epsilon}\|v\|.
\end{equation}
Then there exists a family $\{\mc{H}_x\}_{x\in X}$ of $C^r$ diffeomorphisms $\mc{H}_x\colon \mc{W}(x)\to T_x\mc{W}(x)$ satisfying $\mc{H}_x(x)=0$ and $D_0\mc{H}_x=\Id$ such that for each $x\in X$,
\begin{equation}\label{eqn:normal_forms_linear_dynamics}
D_xf\vert_{T\mc{W}(x)}=\mc{H}_{f(x)}\circ f\circ \mc{H}_x^{-1}\colon T_x\mc{W}\to T_{f(x)}\mc{W}(f(x))
\end{equation}

The maps $\mc{H}_x$ restricted to balls of uniform radius depend continuously on $x\in X$ in the $C^{\lfloor r \rfloor}$ topology and have derivative that is uniformly $r-\lfloor r\rfloor$ H\"older. Further, for each $y\in \mc{W}_x$, the map $\mc{H}_y\circ \mc{H}_x^{-1}\colon T_x\mc{W}\to T_y\mc{W}$ is affine.

\end{thm}

We now upgrade the regularity of the $\mc{S}^F$ foliation. The idea behind the subsequent proofs is to use that that the normal forms coordinates are uniformly bilipschitz on small balls. This allows us to compare the rate at which curves shrink in both the normal forms coordinates and in the ambient manifold.

\begin{prop}\label{prop:slow_is_C_2}
Suppose that $F$ is a $C^{2+}$ Anosov diffeomorphism that is $U_{\lambda}$-framed as in definition \ref{defn:U_lambda_framed}.
The $\mc{S}^F$ foliation is uniformly $C^{2+}$ when restricted to $\mc{W}^{u,F}$ leaves.
\end{prop}

\begin{rem}
The following proof uses the normal forms coordinates. One could also prove this by using that the cocycle holonomies restricted to leaves are $C^{1+}$ and show that the $e_1^F$ distribution is invariant under these holonomies by studying the map $F^n H_{f^{-n}xf^{-n}y}F^{-n}$, where $H_{xy}$ denotes the cocycle holonomy between two points $x$ and $y$ in the same unstable leaf.
\end{rem}

\begin{proof}
To see this note that if $\gamma\colon [0,1]\to M$ is a $C^1$ curve tangent to the $e_1^F$ distribution then 
\begin{equation}\label{eqn:points_separate_slow}
d_{\mc{W}^{u,F}}(F^{-n}\gamma(0),F^{-n}\gamma(1))\le C\lambda^{-n},
\end{equation}
because $DF^{-1}(e_1^F)=\lambda^{-1}e_1^F$.
 Note that for any $x\in M$ that the Riemannian distance on $\mc{W}^{u,F}_{loc}(x)$ and the Euclidean distance on $\mc{W}^{u,L}$ with respect to the normal forms chart on $\mc{W}^{u,F}$ are uniformly bilipschitz as the normal forms depend continuously on the basepoint in the $C^{1}$ topology.

Note that $T_x\mc{W}^{u,F}$ comes equipped with the framing $[e_1^F,e_2^F]$, which gives coordinates on it as a manifold and as a vector space. We will always use this framing to express differentials of maps as matrices. We also fix a metric on $T_x\mc{W}^{u,F}$ that makes $e_1^F$ and $e_2^F$ orthonormal.

Note that 
\[
\begin{bmatrix}
\lambda & 1 \\
0 & \lambda
\end{bmatrix}^{-n}=\lambda^{-n}
\begin{bmatrix}
1 & -n\lambda\\
0 & 1
\end{bmatrix}.
\]
Hence, by equation \eqref{eqn:normal_forms_linear_dynamics}, $F^{-n}$ when viewed as a map $T_x\mc{W}^{u,F}\to T_{F^{-n}(x)}\mc{W}^{u,F}$ in normal forms coordinates is
\begin{equation}\label{eqn:composition_of_charts_and_map}
\mc{H}_{F^{-n}(x)}F^{-n}\mc{H}_x^{-1}=
\lambda^{-n}\begin{bmatrix}
1 & -n\lambda\\
0 & 1
\end{bmatrix}.
\end{equation}
Suppose that $\gamma\colon [0,1]\to T_x\mc{W}^{u,F}$ is non-constant in the $e_2^F$ direction. Then it is immediate from equation \eqref{eqn:composition_of_charts_and_map} that there exists $C_1$ such that
\begin{equation}\label{eqn:curve_growth_est_1}
d_{T_{F^{-n}(x)}\mc{W}^{u,F}}(\mc{H}_{F^{-n}(x)}F^{-n}\mc{H}_x^{-1}(\gamma(0)),
\mc{H}_{F^{-n}(x)}F^{-n}\mc{H}_x^{-1}(\gamma(1)))\ge C_1n\lambda^{-n}.
\end{equation}

Consider now a curve $\eta$ in $\mc{W}^{u,f}$ that is tangent to the $e_1^F$ distribution. Then as in equation \eqref{eqn:points_separate_slow}, there exists $C_2$ such that 
\[
d_{\mc{W}^{u,F}(F^{-n}(x))}(F^{-n}(\eta(0)),F^{-n}(\eta(1)))\le C_2\lambda^{-n}.
\]
Note that under backwards iteration the curve $F^{-n}\circ \gamma$ lies in a uniformly small ball. There exists $C_3$ such that on such uniformly small balls, $\mc{H}_x$ is $C_3$-bilipschitz. Thus
\begin{equation}\label{eqn:curve_growth_est_2}
d_{T_{F^{-n}(x)}\mc{W}^{u,F}}(\mc{H}_{F^{-n}(x)}F^{-n}(\eta(0)),\mc{H}_{F^{-n}(x)}F^{-n}(\eta(1)))\le C_3C_2\lambda^{-n}.
\end{equation}
Thus $\mc{H}_x\eta$ must lie tangent to the linear $e_1^F$ foliation on $T_x\mc{W}^{u,f}$ because otherwise 
equation \eqref{eqn:curve_growth_est_1} would contradict equation \eqref{eqn:curve_growth_est_2}. 

This implies that the linear foliation of $T_x\mc{W}^{u,F}$ tangent to $e_1^F$ pushes forward by $\mc{H}_x$ to the $\mc{S}^F$ foliation. Thus the $\mc{S}^F$ foliation has the same regularity as $\mc{H}_x$, as it is the image under $\mc{H}_x$ of the $e_1^F$ foliation. As the maps $\mc{H}_x$ are uniformly locally $C^{2+}$, this implies that $\mc{S}^F$ is a uniformly $C^{2+}$ foliation subordinate to $\mc{W}^{u,F}$.
\end{proof}

Note that we may consider holonomies along the $\mc{S}^L$ foliation between distant uniform transversals $T$ and $W$. As the holonomies of the $\mc{S}^L$ foliation are isometries, this implies that the resulting map $T\to W$ is uniformly continuous independent of the distance between $T$ and $W$ along $\mc{S}^L$ leaves. We will now develop a sequence of lemmas and prove that the analogous property holds for $\mc{S}^F$.

For $\theta\in [0,\pi]$, we say that a transversal $W$ is $\theta$-transverse to $\mc{S}^F$, if at every point, we have that $W$ makes an angle of at least $\theta$ with $\mc{S}^F$. 
The following says that transversals to the $\mc{S}^F$ foliation travel between uniformly distant leaves in uniform time. In order to describe the distance between leaves, we use the notion of the Hausdorff distance between two sets in a metric space $(M,d)$. The \emph{Hausdorff distance} is defined as:
\[
d_H(X,Y)=\max\{\sup_{x\in X}d(x,Y), \sup_{y\in Y}d(y,X)\}.
\]
If $X$ and $Y$ are two sets which both lie inside a leaf $W$ of another foliation $\mc{W}$, then we write
\[
d_H^{\mc{W}}(X,Y)=\max\{\sup_{x\in X}d_{W}(x,Y), \sup_{y\in Y}d_{W}(y,X)\},
\] 
where $d_W$ is the distance along the leaf $W$ with $W$'s pullback metric.

We will prove the following by covering the foliation with tiny charts in which the claim is trivial.

\begin{lem}\label{lem:transversals_travel}
Suppose that $F$ is $U_{\lambda}$-framed. For any $D>0$, $\theta\in (0,\pi/2]$, there exist $C>0$, such that if $x\in \mc{W}^{u,F}(y)$ and $d_H^{\mc{W}^{u,F}}(\mc{S}^F(x),\mc{S}^F(y))<D$, then any $\theta$-uniform transversal from $x$ to $\mc{S}^{F}(y)$ has length at most $C$.
\end{lem}

Before we begin the proof, we introduce some definitions. As a reminder a local product neighborhood for a point in a foliation is a foliation chart on a neighborhood of the point. Let $B^n_r(0)$ denote the ball of $r$ radius in $\R^n$. For what follows, we will work with foliation charts whose domain is of the form $\phi\colon [-a,b]\times B^{n-1}_r(0)\subset \R^n\to M$; the reason for this is so that we can refer to the ``top" and ``bottom" plaques of the chart, by which we mean $\phi(\{-a\}\times B^{n-1}_r(0))$ and $\phi(\{b\}\times B^{n-1}_r(0))$. As the $\mc{S}^F$ foliation is transversely orientable, we require that this ordering of the end leaves agrees with the transverse orientation on the foliation.

As the $\mc{S}^L$ foliation is Euclidean, it makes perfect sense to speak of the vertical ``distance" between two leaves. This distance is precisely equal to the Hausdorff distance between two leaves, though we refer to it as vertical as we are thinking of it in this geometric way.

\begin{defn}
For $\theta\in (0,\pi/2],t_0>0,\epsilon>0$, a $(\theta,t_0,\epsilon)$-chart for a point $x$ is a local product neighborhood for the $\mc{S}^F$ foliation such $x$ lies in the middle of the bottom plaque of the neighborhood and any unit speed transversal beginning from $x$ that is $\theta$-transverse hits the top leaf of the foliation chart in at most $t_0$-time.  Further, we require that $h$ carries the top and bottom plaques of the chart to leaves of the $\mc{S}^L$ foliation at vertical distance at least $\epsilon$ from each other.
\end{defn}

\begin{proof}[Proof of Lemma \ref{lem:transversals_travel}.]

To begin, we show that for every $\theta\in (0,\pi]$ there exist $t_{\theta},\epsilon>0$ such that every point $x\in \mathbb{T}^4$ is contained in the bottom leaf of a $(\theta,t_{\theta},\epsilon)$-chart. To prove this we essentially use two claims, one about the uniform continuity of $h$ and the other about the uniform $C^1$-ness of $\mc{S}^F$.
\begin{enumerate}
\item
Note that because $h$ is uniformly continuous, that for any $\epsilon>0$ there exists $\epsilon_L>0$ such that if $P_1$ and $P_2$ are two bounded plaques of the $\mc{S}^F$ foliation in the same $\mc{W}^{u,F}$ leaf with $d(P_1,P_2)>\epsilon$ then $h(P_1),h(P_2)$ lie in leaves of the $\mc{S}^L$ foliation at vertical distance at least $\epsilon_L>0$ apart. 
\item
Because the $\mc{S}^F$ foliation is uniformly $C^1$, for every sufficiently small $\epsilon>0$, there exist $t_{\epsilon}>0$ and $\delta>0$ such that every point $x$ is contained in the middle of the bottom plaque of a product neighborhood where a $\theta$-transversal beginning from $x$ hits the top plaque in at most most $t_{\epsilon}$ time and, in addition, $d(P_1,P_2)>\delta$.
\end{enumerate}

By combining these two statements, we obtain that for any $\theta>0$, there exist $t_{\theta},\epsilon_{\theta}>0$ and a cover of the $\mc{S}^F$ foliation by $(\theta,t_\theta,\epsilon_\theta)$-charts. Further, note that as the $\mc{S}^F$ foliation is transversely orientable, we may choose these charts to agree with this transverse orientation.

We now show that uniform transversals uniformly travel ``vertically." Consider a $\theta$-transverse curve $\gamma$ to the $\mc{S}^F$ foliation passing through a point $x$.  Let $B$ be a $(\theta,t_{\theta},\epsilon_{\theta})$-neighborhood for $x$ and let $P_1,P_2$ be the top and bottom plaques of $B$. Note that the vertical distance between $h(P_1)$ and $h(P_2)$ is at least $\epsilon_\theta$. Thus we see that as $\gamma$ traverses $B$, the vertical height of $h(\gamma)$ increases by $\epsilon_\theta$. This implies that $h(\gamma(0))$ and $h(\gamma(nt_{\theta}))$ are vertical distance at least $n\epsilon_\theta$ apart.

If $\mc{S}^F(x)$ and $\mc{S}^F(y)$ are at Hausdorff distance at most $D$ apart, then because $h$ is a quasi-isometry, there exists $D'>0$ such that $d_H^{\mc{W}^{u,L}}(h(\mc{S}^F),h(\mc{S}^F))<D'$. Thus we see that a unit-speed uniform $\theta$-transversal starting at the $\mc{S}^F(x)$ leaf will reach the $\mc{S}^F(y)$ leaf in at most $\lceil D'\epsilon_\theta^{-1}t_{\theta} \rceil$ time. The claim follows.
\end{proof}

We may now prove the following.

\begin{lem}\label{lem:uniform_C_1_holonomies}
Suppose that $F$ is $U_{\lambda}$-framed, then the holonomies of the $\mc{S}^F$ foliation are uniformly $C^1$. 
That is, if $W$ and $T$ are two uniform $C^1$ holonomy related transversals, then the map $\Pi_{T,W}\colon T \to W$ has uniformly continuous Jacobian, which is bounded independent of the transversals. 
\end{lem}

The approach of the following proof is to use normal forms coordinates to linearize the $\mc{S}^F$ foliation and study how transversals to the $\mc{S}^F$ foliation are stretched by its holonomies. Because $\mc{S}^F$ is intertwined with the $\mc{S}^L$ foliation, which has isometric holonomies, the holonomies of $\mc{S}^F$ can only stretch curves a limited amount transverse to the $\mc{S}^F$ foliation. This observation then implies the result.
\begin{proof}

By Lemma \ref{lem:transversals_travel}, there exists $C_1>0$ such that if $T$ is a transversal to the $\mc{S}^F$ foliation of length less than $1$, and $W$ is a uniform transversal to the $\mc{S}^F$ foliation, then $\Pi_W(T)$ has length less that $C_1$. This is immediate because $\Pi_W(T)$ is itself a uniform transversal.

Let $\mc{H}_x\colon \mc{W}^{u,F}(x)\to T_x\mc{W}^{u,F}(x)$ be the normal forms coordinates as in Theorem \ref{thm:normal_forms}. On balls of uniform radius $r_0\gg C_1$, the normal forms coordinates are uniformly $D$-bilipschitz for some $D\ge 1$. By that theorem, we have in addition that if $x\in \mc{W}^{u,F}(y)$, then the map $\mc{H}_y\circ \mc{H}_x^{-1}\colon T_x\mc{W}^{u,F}\to T_y\mc{W}^{u,F}$ is affine. In fact, its differential, with respect to the $[e_1^F,e_2^F]$ framing, is of the form
\begin{equation}\label{eqn:differential_coordinate_change}
\begin{bmatrix}
a & b \\
0 & c
\end{bmatrix},
\end{equation}
because the normal forms intertwine the $\mc{S}^F$ foliation with the linear foliation of $T_x\mc{W}^{u,F}$ tangent to $e_1^F$ as we saw in the proof of Proposition \ref{prop:slow_is_C_2}.

Fix a small number $0<\epsilon\ll r_0$.
Suppose that $I$ is a uniform transversal to the $\mc{S}^F$ foliation through the point $x$ that has length at least $\epsilon$ but length less than $r_0$.

 Let $I_x=\mc{H}_x\circ I$. Then $I_x$ is a curve in $T_x\mc{W}^{u,F}$. Suppose that $d{I}_x/dt$ is $\alpha(t)e_1^F+\beta(t)e_2^F$. Because $I_x$ has length at least $D^{-1}\epsilon$ and is uniformly transverse to $\mc{S}^F$, there exists a uniform constant $C_2>0$ such that 
\begin{equation}\label{eqn:int_of_beta}
\int \beta(t)\,dt\ge C_2 D^{-1}\epsilon.
\end{equation}

Consider a point $y\in \mc{W}^{u,F}$ such that $y\in \mc{S}^F(I)$. If we view $I_x$ in the $T_y\mc{W}^{u,F}$ normal coordinates, then $I_x$ has differential $D\mc{H}_y\mc{H}_x^{-1}\circ I_x$, which by equation \eqref{eqn:differential_coordinate_change} has differential with $e_2^F$ component equal to
\[
c\beta(t)e_2^F.
\]
Let $J$ be a curve through $0$ in $T_y\mc{W}^{u,F}$ tangent to $e_2^F$. Note that $\mc{H}_y^{-1}J$ is a uniform transversal to the $\mc{S}^F$ foliation.
As a curve in $T_y\mc{W}^{u,F}$, $\Pi_{J}\mc{H}_y\mc{H}_x^{-1}I_x$ has tangent $c\beta(t)e_2^F$. Thus by equation \eqref{eqn:int_of_beta},
\[
\text{len}(\Pi_J\mc{H}_yI)\ge cC_2D^{-1}\epsilon.
\]
But by the uniform $D$-bilipschitzness of normal coordinates, this implies that 
\[
\text{len}(\Pi_{\mc{H}_yJ}I)\ge cC_2D^{-2}\epsilon.
\]
But by the first paragraph of this proof, the length of this curve is at most $C_1$. Thus
\[
C_1\ge cC_2D^{-2}\epsilon, 
\]
so
\[
c\le C_1C_2^{-1}D^2\epsilon^{-1}.
\]
Note that this estimate is independent of $x$ and $y$. This implies that in normal forms coordinates that the holonomies are uniformly $C^1$ because the entry $c$ in equation \eqref{eqn:differential_coordinate_change} is uniformly bounded. But this implies that the same result holds for the $\mc{S}^F$ holonomies because the normal forms coordinates are uniformly locally $C^{2+}$.
\end{proof}

The previous discussion shows that the distribution defined by the vector field $e_1^F$ is $C^{1+}$, however, this is not the same as the vector defining the distribution being $C^{1}$. We now upgrade the regularity by using regularity for solutions of Livsic equations following 
\cite[Thm 2.2]{nitica1998regularity}.

\begin{lem}
Suppose that $F$ is $U_{\lambda}$-framed. Then the vector field $e_1^F$ is uniformly $C^{1+}$.
\end{lem}
\begin{proof}
Let $\wt{e}_1^F$ be a $C^{1+}$ rescaling of $e_1^F$, which exists because $e_1^F$ is tangent to a $C^{1+}$ distribution by Lemma \ref{prop:slow_is_C_2}.

Let $[e_1^F]$ denote the subspace tangent to $e_1^F$. Then consider the $C^{1+}$ cocycle $A\colon \mathbb{T}^4\to \R$ arising from $DF\vert_{[e_1^F]}\colon \mathbb{T}^4\to \R$ that sends $x$ to the number $\eta(x)$ such that $DF\wt{e}_1^F=e^{\eta(x)}\wt{e}_1^F$.  Note that if $p$ is a periodic point of period $p$, then $\sum_{i=1}^n A(F^i(p))=n\ln \lambda$. 

We claim that in fact $A$ is cohomologous to constant via a transfer function $\phi\colon \mathbb{T}^4\to \R$ that is uniformly $C^{1+}$ along $\mc{W}^u$ leaves. This follows from the argument in \cite[Thm 2.2]{nitica1998regularity}, which we now sketch. If we view the function $A$ as defining a cocycle $\mc{A}$ on the space $\mathbb{T}^4\times \R$, then this cocycle is partially hyperbolic because $\R$ is abelian. Hence the unstable foliations lift to $\mathbb{T}^4\times \R$ and these are uniformly $C^{1+}$ foliations $\mc{W}^{u,\mc{A}}$. If $\mc{B}$ is another uniformly $C^{1+}$ $\R$-valued cocycle along $\mc{W}^{u,F}$ leaves arising from a function $B\colon \mathbb{T}^4\to \R$, then we similarly get uniformly $C^{1+}$ foliations $\mc{W}^{u,\mc{B}}$. One can then check that the transfer function carries $\mc{W}^{u,\mc{A}}$ leaves to $\mc{W}^{u,\mc{B}}$ leaves. As each of these foliations has uniformly $C^{1+}$ leaves, we see by the implicit function theorem that a transfer function must be uniformly $C^{1+}$ as well. 

Using the function $\phi$ to rescale $\wt{e}_1^F$, we may find a new distribution $\hat{e}_1^F$ that is uniformly $C^{1+}$ on $\mc{W}^{u,F}$ leaves and satisfies $DF\hat{e}_1^F=\lambda \hat{e}_1^F$. It suffices to now check that $e_1^F$ coincides with $\hat{e}_1^F$.  Note that $\hat{e}_1^F=\beta(x)e_1^F$ for some continuous function $\beta$. Then computing $DF\beta(x)e_1^F$ in two different ways we find:
\[
\lambda\beta(x)e_1^F(\sigma(x))=\lambda \beta(F(x))e_1^F,
\]
hence $\beta$ is constant on orbits and hence is constant. Thus in fact $\hat{e}_1^F=Ce_1^F$ and so $e_1^F$ was $C^{1+}$ all along.
\end{proof}

\begin{lem}\label{lem:F_intertwines_flow}
Suppose that $F$ is $U_{\lambda}$-framed and let 
 $\phi^t$ be the flow along the $e_1^F$ vector field, i.e. tangent to slow leaves. Then 
\begin{equation}
	\phi^t=F^{n}\circ \phi^{\lambda^{-n}t}\circ F^{-n}.
\end{equation}
\end{lem}
\begin{proof}
Both sides of the equality are flows of $C^1$ vector fields, so it suffices to check that they have the same generating field, i.e. the right hand side is generated by $e_1^F$.

	If we write $\phi^{t}$ in coordinates as $(x,y)\mapsto (x,y)+te_1^F+O(t^2)$, then we see that the composition takes
\begin{align*}
(x,y)&\mapsto F^{-n}(x,y)\\
&\mapsto F^{-n}(x,y)+t\lambda^{-n}e_1^F+O((\lambda^{-n}t)^{1+\alpha})\\
&\mapsto (x,y)+te_1^F+O(n\lambda^{-n\alpha}t^{1+\alpha})
\end{align*}
as desired.
\end{proof}

The following proof is the only place where we use that the Jordan periodic data of $F$ is exactly that of $L$. It seems that in general, if $G$ has Jordan periodic data $\alpha_G$ and $\int \alpha_G\,d\vol>0$, then the following proof can still be carried out. It is unknown to the author whether this property necessarily holds for the Jordan data of all Anosov diffeomorphisms with the periodic data of $L$.

\begin{lem}\label{lem:uniform_transversals_flow}
Let $F$ be $U_{\lambda}$-framed.
Let $\phi^t$ be the flow along $e_1^F$ and suppose that $F$ has the same Jordan data as $L$. Then the flow $\phi^t$ carries uniform $C^1$ transversals to $\mc{S}^F$ to  uniform $C^1$ transversals to $\mc{S}^F$.
\end{lem}

\begin{proof}
In order to do this we will study the differential of the flow $\phi^t$ in the $[e_1^F,e_2^F]$ framing. We will do so by first studying the differential of the flow $\phi^t$. We write all differentials below with respect to the $[e_1^F,e_2^F]$ framing.

Because $e_1^F$ is a $C^{1+}$ vector field, we have that the flow map $\phi^t(x)$ is $C^{1+}$ in time and initial conditions.\footnote{The author is unaware of a reference in the literature for this claim; however, one can show this by reading, for example, \cite{izzo1999} and keeping track of the H\"older constant during Picard iteration.} Thus in a coordinate chart we may write $D\phi^t=\Id+O(t^{\sigma})$ for the differential of the flow of time $t$ for some $0<\sigma<1$. We will only be interested in this differential for increasingly short times $\lambda^{-n}t$.

We now find an expression for $D\phi^t$ in terms of the $[e_1^F,e_2^F]$ framings. We can always choose our coordinates to be tangent to the vector fields $[e_1^F,v]$ where $v$ is some smooth transverse field to $e_1^F$. Then $v=\alpha e_1^F+\beta e_2^F$ and we can insist that both $\alpha$ and $\beta$ are $C^{\sigma}$-H\"older and that $\beta$ is uniformly bounded below by transversality. For ease of notation let $x$ be the initial point and $y=\phi^{\lambda^{-n}t}$.

 To change into the $[e_1^F,e_2^F]$ framing, we must conjugate $\Id+O(\lambda^{-n\sigma}t^{\sigma})$ as follows.
\begin{align*}
&\begin{bmatrix}
1 & -\alpha(y)\beta^{-1}(y)\\
0 & \beta^{-1}(y)
\end{bmatrix}
\left[\Id+O(\lambda^{-n\sigma}t^{\sigma})\right]
\begin{bmatrix}
1 & \alpha(x)\\
0 & \beta(x)
\end{bmatrix}\\
&=
\begin{bmatrix}
1 & -\alpha(y)\beta^{-1}(y)\\
0 & \beta^{-1}(y)
\end{bmatrix}
\begin{bmatrix}
1 & \alpha(x)\\
0 & \beta(x)
\end{bmatrix}
+O(\lambda^{-n\sigma}t^{\sigma})\\
&=\begin{bmatrix}
1 & \alpha(x)-\alpha(y)\beta^{-1}(y)\beta(x)\\
0 & \beta^{-1}(y)\beta(x).
\end{bmatrix}
+O(\lambda^{-n\sigma}t^{\sigma}).\\
\end{align*}
Note that $d(x,y)=O(\lambda^{-n}t)$. As $\beta$ is bounded below, this implies that \[\beta^{-1}(y)\beta(x)=1+O(\lambda^{-n\sigma}t^{\sigma}).\] Using a H\"older estimate on $\alpha$ and cancelling, gives that with respect to the $[e_1^F,e_2^F]$ framings,
\[
D\phi^{\lambda^{-n}t}=\Id+O(t^{\sigma}\lambda^{-n\sigma}),
\]
for some $0<\sigma<1$.

Hence with respect to the framings, using the relation $\phi^t=F^n\circ \phi^{\lambda^{-n}t}\circ F^{-n}$, we find
\[
D\phi^t=
\begin{bmatrix}
\lambda^n & n\lambda^{n-1}\\
0 & \lambda^n
\end{bmatrix}\pez{\Id +O(t^{\sigma}\lambda^{-\sigma n})}
\begin{bmatrix}
\lambda^{-n} & -n\lambda^{-n+1}\\
0 & \lambda^{-n}
\end{bmatrix}.
\]
But this is converging to $\Id$ as $n\to \infty$. Thus we see that for all times $t$, $D\phi^t$ is the map that sends $e_1^F\mapsto e_1^F$ and $e_2^F\mapsto e_2^F$ at corresponding base points. 

In particular, this implies the result because it shows that curves that are uniformly transverse to the flow direction $e_1^F$ remain uniformly transverse to $e_1^F$ and are not distorted in length as the fields $e_1^F$ and $e_2^F$ are uniformly bounded above and below in length.
\end{proof}

\begin{rem}
It is perhaps not surprising that the $e_2^F$ field is preserved by $\phi^t$ because the frames $[Ce_1^F,Ce_2^F]$ are the only frames that present the cocycle as a constant Jordan block.
\end{rem}

\subsection{The conjugacy is Lipschitz}

Now that we have established stronger properties of the $\mc{S}^F$ foliation, we will use these to show that $h$ is Lipschitz.

\begin{lem}\label{lem:lipschitz_conj}
Suppose that $F$ is $C^{2+}$ and $U_{\lambda}$-framed. Then $h$ is uniformly Lipschitz restricted to each unstable leaf.
\end{lem}

The proof of Lemma \ref{lem:lipschitz_conj} is based on exhibiting $h$ as the uniform limit of a sequence of uniformly $C^1$ functions. The main difficulty we encounter is that we cannot establish regularity of $h$ in a single step. Instead, we establish regularity of $h$ incrementally by studying its regularity in some directions before others. As an ansatz, suppose we knew that $h$ was differentiable with
\[
D_zh=\begin{bmatrix}
a(z)& b(z)\\
0 & c(z)
\end{bmatrix},
\]
with respect to the framings $[e_1^L,e_2^L]$ and $[e_1^F,e_2^F]$. It is easy to deduce that $a(z)=C$ for some number $C$ because $h$ intertwines the $\mc{S}^L$ and $\mc{S}^F$ foliations. It is straightforward to check that $c(z)=C$ as well by using that $b(z)$ is uniformly bounded. Below we study the ``differential" of $h$ in this manner. We first study the possible diagonal ``entries" of the derivative of $h$ and then study the ``upper right-corner" of $Dh$. This approach is substantially complicated because, as far as the author can tell, partial derivatives do not make sense invariantly for functions that are not $C^1$. However, we can still recover the information we need by studying a derivative normal to the foliations $\mc{S}^L$ and $\mc{S}^F$, which we introduce in Definition \ref{defn:normal_derivative}.

This section relies on the following lemma of de la Llave. The version we give below is a slight rephrasing of the original adapted to our setting. See \cite{gogolev2020local} for a recent example of the use of this lemma in a similar context.
\begin{lem}\cite[Thm. 2.1]{delallave2002rigidity}\label{lem:delallave}
Let $f,g$ be $C^1$ Anosov diffeomorphisms of a closed manifold $M$. Let $h$ be a homeomorphism of $M$ such that 
\begin{enumerate}
\item
$h\circ f=g\circ h$.
\end{enumerate}
Let $k$ be a map---not necessarily invertible or continuous---such that:
\begin{enumerate}
\item[(2)]
$k(\mc{W}^{u,f}_x)\subset \mc{W}^{u,g}_{h(x)}$, and
\item[(3)]
$\sup_x d_{\mc{W}^u}(k(x),h(x))<\infty.$
\end{enumerate}
Then
\[
h(x)=\lim_{n\to \infty} g^{-n}\circ k\circ f^{n}(x),
\]
and the limit is reached uniformly with respect to the distances $d_{\mc{W}^{u,f}}$ and $d_{\mc{W}^{u,g}}$.
\end{lem}
As the convergence in Lemma \ref{lem:delallave} is uniform, we can use it to study whether the conjugacy is Lipschitz by exhibiting the conjugacy as a uniform limit of uniformly Lipschitz functions.

In the sequel we will write $Dh(e_1^L)$ to mean the derivative of $h$ in the direction of $e_1^L$ along the foliation $\mc{S}^L$. By Proposition \ref{prop:smooth_on_slow_leaves}, this derivative exists.
Note that when we make use of this notation we are not asserting that $h$ is differentiable.

\begin{lem}\label{lem:partial_x_h_1}
There exists $C_1\neq 0$ such that $Dh{e_1^L}=C_1e_1^F$. Without loss of generality, we may assume that $C_1=1$.
\end{lem}
\begin{proof}
By Proposition \ref{prop:smooth_on_slow_leaves}, we know that this derivative exists and that 
\[
Dh(e_1^L)=c(x)e_1^F
\]
for some H\"older function $c(x)$.
Because
\[
F\circ h=h\circ L,
\]
we can differentiate this equation along $\mc{S}^L$ on both sides and use the chain rule to find that
\[
\lambda c(x)e_1^F=\lambda Dh(e_1^L)=Dh(e_1^L)\lambda=c(L(x))\lambda e_1^F.
\]
Thus $c(x)$ is constant on orbits. Hence as $c(x)$ is continuous it is constant. 
To ensure that $C_1=1$, we replace $[e_1^F,e_2^F]$ with $[C_1e_1^F,C_1e_2^F]$. Note that this does not change the Jordan periodic data or the function $\alpha$ representing the data in this framing.
\end{proof}

We now show our first smoothing lemma.

\begin{lem}\label{lem:first_smoothing}
Suppose that $F$ is $C^{2+}$ and $U_{\lambda}$-framed. 
There exists a function $h_0\colon M\to M$ such that, writing $h_x$ for the restriction $h_0\vert_{\mc{W}^{u,L}(x)}$, we have that:
\begin{enumerate}
\item
$h_0$ intertwines the $\mc{W}^{u,L}$ and $\mc{W}^{u,F}$ foliations as well as the $\mc{S}^L$ and $\mc{S}^F$ foliations;
\item
$h_x$ is $C^1$ and is uniformly $C^0$ close to $h$ with respect to $d_{\mc{W}^u}$;
\item
With respect to the framings $[e_1^L,e_2^L]$, $[e_1^F,e_2^F]$, the differential of $h_0$ is
\[
\begin{bmatrix}
1 & b(x)\\
0 & a(x)
\end{bmatrix},
\]
where $a$ is a uniformly bounded continuous function on each leaf and the uniformity is independent of the leaf.
\end{enumerate}
\end{lem}

Before beginning the proof we describe how one studies the regularity in a situation like this. 
If we want to show that a function $k\colon \mc{W}^{u,L}(x)\to \mc{W}^{u,F}(h(x))$ is uniformly $C^1$, then it suffices to check that this is true in a particular family of charts adapted to the $C^1$ norms on these manifolds.
While $\mc{W}^{u,L}(x)$ admits a chart $\psi^L\colon \R^2\to \mc{W}^{u,L}(x)$ that is an isometry, and hence uniformly $C^1$, the leaf $\mc{W}^{u,F}(h(x))$ does not a priori have such a nice parametrization.
However, $\mc{W}^{u,F}$ does admit a \emph{uniform} family of charts, which we now describe.
Let $\phi$ be the flow of the $C^{1+}$ vector field $e_1^F$ along the leaves of $\mc{S}^F$. Note that for any fixed $t_0>0$ the maps $\phi^t$ for $t\in [-t_0,t_0]$ are uniformly $C^{1+}$.
Fixing numbers $N, t_0>0$ and a parameter $\theta>0$, if $\gamma(t)\colon (-\delta,\delta)\to \mc{W}^{u,F}(h(x))$ is any unit speed curve of length between $N$ and $2N$ that makes an angle of at least $\theta$ with $e_1^F$ distribution, then the map $\psi_{\gamma}^F\colon (-\delta,\delta)\times (-t_0,t_0)\to \mc{W}^{u,F}(h(x))$ defined by
\[
(t_1,t_2)\mapsto \phi^{t_2}\gamma(t_1),
\]
is a parametrization of a subset of $\mc{W}^{u,F}$. In addition $\psi_{\gamma}^F$ is a $C^{1+}$ chart for the $\mc{S}^F$ foliation. 

In particular, given this uniform family of charts, we may study regularity in the following way.
 If we want to show that a map $k\colon \mc{W}^{u,L}(x)\to \mc{W}^{u,F}(h(x))$ is uniformly $C^1$, it suffices to show that for each uniform pair of charts $\psi_{\gamma}^F$ and $\psi^L$, we have that 
\begin{equation}\label{eqn:uniform_chart}
(\psi_{\gamma}^F)^{-1}\circ k\circ \psi^L,
\end{equation}
is uniformly $C^1$ as a map defined on a subset of $\R^2$. In the proof that follows, we will use this family of charts in order to study the regularity.

As mentioned before, because the leaves of the $\mc{W}^{u,L}$ foliation are Euclidean they admit global uniformly smooth charts. Specifically, on the $\mc{W}^{u,L}$ foliation we exclusively work with the global charts of the form 
\[
(t_1,t_2)\mapsto x+t_1e_1^L+t_2e_2^L.
\]
We construct these charts around a specific transversal to the $\mc{S}^L$ foliation of the form $T_0\colon t\mapsto x+te_2^L$ in a manner analogous to \eqref{eqn:uniform_chart}. We similarly denote such a chart by $\psi_{T_0}^L$. In addition, these transversals $t\mapsto x+te_2^L$ define a linear foliation, which we denote by $\mc{F}^L$.

\begin{proof}[Proof of Lemma \ref{lem:first_smoothing}.]

We give the construction for a single leaf $\mc{W}^{u,L}(x)$. The full result follows by repeating the construction on each leaf.

Fix a uniform $C^1$ transversal $T_F\colon \R\to \mc{W}^{u,F}(h(x))$ to the $\mc{S}^F$ foliation.  To see that such a transversal exists, first note that there exist uniformly H\"older topological transversals, such as $h(T_0)$, where $T_0$ is a leaf of $\mc{F}^L$. Locally, one may mollify this transversal to obtain a uniformly $C^1$ curve $T_F$ that stays within a bounded distance of $h(T_0)$ and is uniformly transverse to the $\mc{S}^F$ foliation. In particular, we use that $h(T_0)$ and $T_F$ are uniformly $C^0$ close as maps $\R\to \mc{W}^{u,F}(x)$.

Letting $T_0\colon \R\to \mc{W}^{u,L}$ be a transversal to $\mc{S}^L$ tangent to $e_2^L$, define $P\colon \R\to \R$ by $T_F^{-1}\circ \Pi_{T_F}\circ h\circ T_0$.
Note that $P$ is uniformly H\"older because it
is the composition of uniformly H\"older maps.

Along the $\mc{F}^L$ leaf $T_0$, we can mollify $P$ as follows to get a function $\wt{P}_{T_0}$. Fixing any smooth bump function $\sigma$ on $\R$, we define 
\[
\wt{P}_{T_0}(t)\coloneqq \int P(t-x)\sigma(x)\,dx.
\]
Note that $\wt{P}_{T_0}$ is $C^0$ close to $P$ and that $T_F\circ \wt{P}_{T_0}$ is uniformly $C^0$ close to $T_F\circ P_{T_F}=\Pi_{T_F}\circ h\circ T_0$. But as $T_F$ is $C^0$ close to $h\circ T_0$, this is uniformly $C^0$ close to $h\circ T_0$ and hence $T_F\circ \wt{P}_{T_0}$ is uniformly $C^0$ close to $h\circ T_0$ as maps $\R\to \mc{W}^{u,F}$.

We now define the map $h_x$.
Define $\hat{h}_2$ to be the map $T_F\circ \wt{P}_{T_0}$. 
As before, let $\phi$ be the flow along the $e_1^F$ vector field tangent to the $\mc{S}^F$ foliation. In the coordinates $(t_1,t_2)=x+t_1e_1^L+t_2e_2^L$ on $\mc{W}^{u,L}$ define
\begin{equation}\label{eqn:y_smoothed_defn}
h_x\colon (t_1,t_2)\mapsto \phi^{t_1}(\hat{h}_2(t_2)).
\end{equation}

We now check that $h_x$ has the required regularity properties. From the definition it is immediate that $h_x$ is $C^1$. Next, we claim that $h_x$ is $C^0$ close to $h$. From earlier, we have that $h(T_0(t))$ is uniformly close to $T_F\circ \wt{P}_{T_0}(t)$. We claim that independent of $s$
\[
\phi^s(T_F\circ \wt{P}_{T_0}(t))
\text{ is uniformly close to }
h(s,t)=\phi^s(h(T_0(t))
\]
This follows because the flow $\phi^s$ is a uniform quasi-isometry independent of the value of $s$: this follows because $h$ is a quasi-isometry and intertwines the flow $\phi^s$ with the linear flow along $\mc{S}^L$, which is isometric.

We now turn to the differentiability properties of $h_x$. From the definition of $h_x$ in equation \eqref{eqn:y_smoothed_defn}, $h_x$ is manifestly $C^1$ and $Dh_x(e_1^L)=e_1^F$. It just remains to check that when written with respect to the bases $[e_1^L,e_2^L]$ and $[e_1^F,e_2^F]$ that the bottom right corner of $Dh_x$ is uniformly bounded.

We begin by checking the corresponding claim for a pair of uniform charts. For a uniform transversal $\gamma$ to $\mc{S}^F$, we  have uniform charts $\psi_{\gamma}^F$ and $\psi^L$ and may write $h_x$ as $(h_1,h_2)$ in these charts.
We claim that in such uniform charts, the derivative of $h_x$ looks like 
\begin{equation}\label{eqn:differential_wrong_coords}
\begin{bmatrix}
1 & b\\
0 & a
\end{bmatrix},
\end{equation}
where $a$ is a uniformly bounded continuous function. 
 To see this note first that $h_2$ in these charts is equal to $\gamma^{-1}\circ \Pi_{\gamma}\circ T_F\circ \wt{P}_{T_0}$. By Lemma \ref{lem:uniform_C_1_holonomies}, the map $\Pi_{\gamma}$ is uniformly $C^1$ as a map $T_F\to \gamma$ as these are both uniform transversals. 
Thus as the remaining maps in the composition are uniformly $C^1$ so is the entire composition.
Hence $a$ is uniformly bounded.

Getting uniformity for the specific framing in the lemma requires only slightly more work.
The matrix \eqref{eqn:differential_wrong_coords} is written with respect to the chart coordinate frames, which are vector fields $[e_1^L,e_2^L]$ and $[e_1^F,\partial_y]$, where $\partial_y$ depends on the chart $\psi_{\gamma}^F$. 
The field $e_2^F$ that we are actually interested in is a linear combination of $e_1^F$ and the chart coordinate $\partial_y$, i.e. $e_2^F=ce_1^F+d\partial_y$ where $c$ and $d$ are uniformly H\"older and $d$ is uniformly bounded away from $0$ because in uniform charts the coordinate directions are uniformly transverse.
Thus written with respect to bases $[e_1^L,e_2^L]$ and $[e_1^F,e_2^F]$, $Dh_x$ has the form:
\[
\begin{bmatrix}
1 & *\\
0 & ad^{-1}
\end{bmatrix}.
\]
Thus the derivative is uniformly bounded as required.
\end{proof}

\subsection{Normal Differentiability}

For a function $A\colon \R^2\to \R^2$, one often studies differentiability of $A$ by studying the partial derivatives of $A$. However, the usual definition of a partial derivative is far from being ``invariantly" formulated. This is particularly true in situations where $F$ is not differentiable but still has partial derivatives. For this reason we introduce a notion of differentiability normal to a foliation.

We say that a foliation $\mc{F}$ has \emph{well-defined} holonomy, if for any two transversals $I$ and $J$ such that $I\subset \mc{F}(J)$, we have that for each $x\in I$ that $\abs{\mc{F}(x)\cap J}=1$. For example, the foliation of $\R^2$ by lines has well-defined holonomy; in particular note that $\mc{S}^L$ and $\mc{S}^F$ both have well defined holonomy when restricted to unstable leaves. A foliation of a closed manifold might not have well defined holonomy. 
When restricted to a coordinate neighborhood, all foliations have well defined holonomy. 
However, we will not work with such a local notion as it is not necessary in our case.

\begin{defn}\label{defn:normal_derivative}
Suppose that $\mc{F}$ and $\mc{G}$ are two $C^1$ foliations with well-defined holonomy. Suppose that $\phi$ is a continuous map intertwining the $\mc{F}$ and $\mc{G}$ foliations. We say that $\phi$ is \emph{normally $C^1$ differentiable} to $\mc{F}$ and $\mc{G}$ if for any pair $T$ and $W$ of smooth transversals to $\mc{F}$ and $\mc{G}$, respectively, with $\phi(T)\subset \mc{G}(W)$, the map $T\to W$ given by $\Pi_{W}\circ \phi$ is $C^1$. (Here, as before, $\Pi_W$ denotes the projection to $W$ along the leaves of $\mc{G}$.)

Note that because $\mc{F}$ and $\mc{G}$ are $C^1$ foliations, the holonomy between different transversals is $C^1$, and hence this notion of differentiability is well defined independent of the choice of transversals.

When we have fixed a pair of transversals $T$ and $W$, we will speak about differentiability of $\phi$ along the transversals $T$ and $W$ normal to the foliations $\mc{F}$ and $\mc{G}$. Further, because the holonomies of $\mc{F}$ and $\mc{G}$ are absolutely continuous with respect to the Lebesgue measure, it makes sense to speak of normal differentiability almost everywhere along a transversal. We may also speak of normal differentiability at a pair of points $x$ and $y\in \mc{G}(\phi(x))$.
\end{defn}

We will restrict the use of this definition to the case where the foliations have dimension $1$ and subfoliate a space of dimension $2$.
In fact, if we have differentiability along such foliations and $v$ and $w$ vectors transverse to $\mc{F}$ and $\mc{G}$ such that the foot of $w$ is the image of the foot of $v$, then there is a well defined number $D_{v,w}\phi$ that we call the normal derivative of $\phi$ between $v$ and $w$, which we define in the following manner. 
Let $T$ and $W$ be two transverals with $dT/dt=v$ and $dW/dt=w$ at corresponding points. We then define $D_{v,w}\phi$ by the equation
\[
D(\Pi_{W}\circ \phi\circ T)(\partial_t)=(D_{v,w}\phi) w.
\]

Having made this definition, we now list a number of claims describing the basic and hopefully intuitive properties of this construction for later use.

\begin{claim}
Suppose that $\phi\colon M_1\to M_2$ intertwines foliations $\mc{F}$ and $\mc{G}$ with well-defined holonomy. Suppose that $v\in T_xM_1\setminus T_x\mc{F}$ and $w\in T_{y}M_2\setminus T_{y}\mc{G}$ such that $D_{v,w}\phi$ exists. Then for any $v'\in TM_1\vert_{\mc{F}(x)}\setminus T\mc{F}(x)$ and $w'\in TM_2\vert_{\mc{G}(\phi(x))}\setminus T\mc{G}(x)$, the normal derivative $D_{v',w'}\phi$ exists.

\end{claim}
\begin{proof}
	Supposing such a pair exists, we can express the derivative defining $D_{v',w'}\phi$ by composing with smooth maps.
	By definition $D_{v,w}\phi$ existing means that if $T$ is a transversal tangent to $w$ at $0$, then
	\[
		\Pi_W\circ \phi\circ T	\colon \R\to W
	\]
	is differentiable at $0$. But this implies that if $T'$ is a transversal tangent to $v'$ at $t=0$ and $W'$ is tangent to $w'$, then
	\[
		\Pi_{W'}\Pi_W\circ \phi\circ T\circ T^{-1}\circ \Pi_T\circ T'=\Pi_{W'}\circ \phi\circ T'
	\]
	is also differentiable at $0$ by the chain rule as we have pre- and post-composed with smooth functions. But $D(\Pi_{W'}\circ \phi\circ T')$ defines $D_{v',w'}\phi$, which we see exists.
\end{proof}

\begin{claim}\label{claim:intertwine_same_deriv}
	Suppose that $\phi\colon M_1\to M_2$ intertwines foliations $\mc{F}$ and $\mc{G}$ with well-defined holonomy.  Suppose that $\psi$ is another map intertwining these foliations. Then if $\psi\colon M_1\to M_2$ satisfies $\psi(x)\in \mc{F}(\phi(x))$ for all $x$, $\psi$ and $\phi$ are both differentiable where either is differentiable and at such points $\phi$ and $\psi$ have the same normal derivative.
\end{claim}

\begin{proof}
For a pair of transversals $W$ and $T$, note that $\Pi_W\circ \phi\circ T=\Pi_W\circ \psi\circ T$; the conclusion is immediate.
\end{proof}

\begin{claim}\label{claim:normal_deriv_coincides_with_old_deriv}
	If $\phi $ is $C^1$ and intertwines foliations as before, then the normal derivative agrees with the usual derivative in the appropriate sense. Namely, If $D\phi (v)=\lambda w$, then $D_{v,w}\phi=\lambda$. 

More broadly, if $T\mc{F}$ and $T\mc{G}$ are the tangents to $\mc{F}$ and $\mc{G}$ then,
\[
\Pi_{w}^{T\mc{G}}DF(v)=(D_{v,w}F)w,
\] 
where $\Pi_{w}^{T\mc{G}}$ denotes the projection onto the subspace spanned by $w$ along $T\mc{G}$.
\end{claim}
\begin{proof}
Extend $v$ and $w$ to a coordinate systems, then in coordinates the derivative of $\Pi_W\circ \phi\circ T$ is precisely the derivative of $\phi$ along $v$ paired with $w$ with respect to the coordinate pairing. The result is independent of the coordinate system because $\phi $ is $C^1$ and both results follow.

For the more general claim, as before, fix $C^1$ transversals $T$ and $W$ tangent to $v$ and $w$. 
The normal derivative $D_{v,w}F$ is equal to the $\lambda$ satisfying, for $C^1$ transversals 
\[
(D\Pi_W\circ \phi\circ T)(\partial_t)=(D_{v,w}\phi)w.
\]
As $\phi$ is $C^1$, we can rewrite this as:
\[
D\Pi_W\circ \phi\circ T=(D\Pi_W)(D\phi\circ T)(\partial_t)=\Pi_w^{T\mc{G}}(DF(v)),
\]
as desired.
\end{proof}

\begin{claim}
Suppose that $\phi$ intertwines two foliations $\mc{F}$ and $\mc{G}$ as before and that
$D_{v,w}\phi$ exists. If $v',w'$ are another pair of vectors with the same respective basepoints, then we can compute $D_{v',w'}\phi$ in the following way. Let $q$ be the tangent to $\mc{F}$ at $v$ and $r$ be the tangent to $\mc{F}$ at the foot of $w$. For a pair of vectors $a$ and $b$ in $\R^2$, write $\Pi^a_b$ for the projection onto $b$ along $a$. Then if $\Pi_{w'}^qw=\lambda w'$ and $\Pi_{v'}^r=\eta v$, then $D_{v',w'}\phi=\lambda\eta D_{v,w}\phi$.
\end{claim}

\begin{proof}
As before, we can rewrite the formula that defines the normal derivative $D_{v',w'}\phi$ into one involving the normal derivative $D_{v,w}\phi$. Namely
\[
\Pi_{W'}\circ \phi\circ T=(\Pi_{W'})\circ \Pi_{W}\circ \phi\circ T\circ (T^{-1}\circ \Pi_T\circ T').
\]
The projections in the statement of the claim are the differentials of the terms in parentheses written in appropriate coordinates.
\end{proof}

\subsection{Normal Differentiability of the conjugacy}

\begin{lem}\label{lem:reg_on_transversals}
For the $\mc{S}^L$ and $\mc{S}^F$ foliations, for any pair $T,W$ of corresponding transversals, we have that $h$ is normally differentiable at almost every point of $T$. In fact, for uniform transversals $T$ and $W$, the induced map $\Pi_{W}h\colon T\to W$ is uniformly Lipschitz.
\end{lem}

\begin{proof}
We will show this only for the case of uniform transversals, as, by the preceding discussion, this implies normal differentiability for all transversals.
We will  show this by studying the convergence in Lemma \ref{lem:delallave}. 
Let $h_n=F^{-n}\circ h_0\circ L^n$ where $h_0$ is as in Lemma \ref{lem:first_smoothing}. We calculate the differential $D_zh_n$ with respect to the $[e_1^L,e_2^L]$, $[e_1^F,e_2^F]$ framings 
\[
Dh_n=
\begin{bmatrix}
\lambda^{-n} & -n\lambda^{-n+1}\\
0 & \lambda^{-n}
\end{bmatrix}
\begin{bmatrix}
1 & c(z)\\
0 & a(z)
\end{bmatrix}
\begin{bmatrix}
\lambda^{n} & n\lambda^{n-1}\\
0 & \lambda^{n}
\end{bmatrix}
=\begin{bmatrix}
1 & c_n(z)\\
0 & a(L^nz)
\end{bmatrix},
\]
where $a$ is a uniformly bounded continuous function and $c_n(z)$ is some continuous function depending on $n$. Fix uniform transversals $T\colon (-1,1)\to \mc{W}^{u,L}(x)$ and $W\colon (-1,1)\to \mc{W}^{u,F}(h(x))$ to the $\mc{S}^L$ and $\mc{S}^F$ foliations, respectively. We can write $\dot{T}$ as $\alpha_L(t)e_1^L+\beta_L(t)e_2^L$ and $\dot{W}$ as $\alpha_F(t)e_1^F+\beta_F(t)e_2^F$. The tangent to $h_n\circ T$ is then
\[
\begin{bmatrix}
\alpha(t)+\beta(t)c_n(T(t))\\
a(L^n(T(t)))\beta(t)
\end{bmatrix}.
\]

Let $\Pi_W$ be the projection to $W$ along $\mc{S}^F$ and let $e^W$ be the tangent field to $W$. 
 Then the differential of $\Pi_W\colon \mc{F}(W)\to W$, with respect to these framings $[e_1^F,e_2^F]$ and $[e^W]$, is 
\[
[0 , \omega]
\]
for some continuous function $\omega$.
We claim that $\omega$ is uniformly continuous.
Note that for a point $y$, $\omega(y)$ is the derivative of the map $\Pi_w\circ \eta$ for $\eta$ a uniform transversal tangent to $e_2^F(y)$; by Lemma \ref{lem:uniform_C_1_holonomies}, this derivative is uniformly bounded.

Thus the derivative of $\Pi_W \circ  h_n\circ T_L \colon \R\to W$ sends
\[
\partial_t \mapsto a(L^n(T(t)))\beta(t)\omega(F^{-n}k(L^n(x))) e_W,
\]
which is uniformly bounded. Thus the sequence of maps $\eta_{n}=\Pi_W\circ h_n\circ T_L$ is uniformly Lipschitz. As the sequence $\eta_{n}$ converges uniformly to the map $\Pi_W\circ h\circ T_L$, we see that this map is uniformly Lipschitz as well, and hence differentiable almost everywhere with uniformly bounded derivative.
\end{proof}

As noted after its definition,  the normal derivative is well defined independent of the transversals used to calculate it.  We now determine the normal derivative to $h$.

\begin{lem}\label{lem:normal_deriv_constant}
There exists $C$ such that for almost every $z$, $D_{e_2^L(z),e_2^F(h(z))}h=C$.
\end{lem}

\begin{proof}
We have that $F\circ h=h\circ L$. Both of these are normally differentiable almost everywhere to the foliations $\mc{S}^F$ and $\mc{S}^L$, thus they have the same normal derivatives. Writing $a(z)$ for the derivative $D_{e_2^L(z),e_2^F(h(z))}h$, by calculating the derivative of $h\circ L$ in two ways, we see that for almost every $z$,
\[
\lambda a(z)=a(L(z))\lambda.
\]
By measurable rigidity for $\R$-valued transfer functions, we see that $a$ coincides almost everywhere with a continuous transfer function. But such a continuous transfer function is constant on orbits hence is constant.
Thus there exists $C$ such that $D_{e_2^L,e_2^F(h(z)}h$ is equal to $C$ almost everywhere.
\end{proof}

We now give another smoothing argument to produce an approximation $k$ to $h$ with even better properties. This time we do not smooth in the normal direction because by Lemma \ref{lem:reg_on_transversals} $h$ already has some regularity in this direction.

\begin{lem}\label{lem:smoothing_with_C}
There exists $C\neq 0$ and a function $h_0\colon M\to M$ such that, writing $h_x$ for the restriction $h_0\vert_{\mc{W}^{u,L}(x)}$, we have that
\begin{enumerate}
\item
$h_0$ intertwines the $\mc{W}^{u,L}$ and $\mc{W}^{u,F}$ foliations as well as the $\mc{S}^L$ and $\mc{S}^F$ foliations;
\item
$h_x$ is uniformly $C^0$ close to $h$;
\item
$h_x$ is uniformly $C^{1}$;
\item
With respect to the framings $[e_1^L,e_2^L]$, $[e_1^F,e_2^F]$, the differential of $h_0$ is
\[
\begin{bmatrix}
1 & b\\
0 & C
\end{bmatrix},
\]
where $b$ is a uniformly continuous bounded function on each leaf and the boundedness is independent of the leaf.
\end{enumerate}
\end{lem}

\begin{proof}
We show how to do the smoothing for a particular leaf $\mc{W}^{u,L}(x_0)$. The result follows by smoothing on each leaf.

As in the proof of Lemma \ref{lem:first_smoothing}, we fix a uniform global transversal $T_F$ to the $\mc{S}^F$ foliation and $T_0$, a parameterization of an $\mc{F}^L$ leaf. Then we define $\omega(t)\in \R$ by
\[
\phi^{\omega(t)}h(T_0(t))\in T_F.
\]
Note that $\omega$ is bounded and uniformly $C^{\alpha}$ for some $\alpha>0$. Thus we may mollify $\omega$ to obtain $\wt{\omega}$, which is uniformly $C^1$ and uniformly close to $\omega$. We now define a new version of $h$, $h_{x_0}\colon \mc{W}^{u,L}(x_0)\to \mc{W}^{u,L}(h(x_0))$ using the $(t_1,t_2)=x_0+t_1e_1^L+t_2e_2^L$ coordinates on $\mc{W}^{u,L}$:
\[
h_{x_0}(t_1,t_2)=\phi^{\wt{\omega}(t_2)+t_1}\Pi_{T_F}\circ h(T_0(t_2)).
\]
We claim that this map satisfies the conclusion of the lemma. The first two claims follow as in the proof of Lemma \ref{lem:first_smoothing}. We will show the remaining claims by studying the derivatives of $h_{x_0}$ in uniform charts.

By Lemma \ref{lem:uniform_transversals_flow}, for all $t\in \R$, $\phi^t(T_F)$ is a uniform transversal to the $\mc{S}^F$ foliation. Hence we may restrict to using uniform chart on $\mc{W}^{u,F}$ formed from such transversals $\phi^t T_F$. This gives us a uniform chart $\psi^F=\psi_{\phi^{t_0}T_F}^F\colon (t_1,t_2)\mapsto \phi^{t_0+t_1}T_F(t_2)$. Let us take the chart $\psi^L\colon (t_1,t_2)\mapsto (x_0+t_0e_1^F)+t_1e_1^F+t_2e_2^F$ on $\mc{W}^L$. Then in these charts the map $(\psi^F)^{-1}\circ h_{x_0}\circ \psi^L=(\wt{\omega}(t_2)+t_1,t_2)$, which is uniformly $C^1$.

It now remains to verify the properties of the differential of $h_{x_0}$. That the differential has block form 
\[
\begin{bmatrix}
1 & b\\
0 &c
\end{bmatrix}
\]
for two continuous function $b$ and $c$ is immediate from its coordinate expression. In particular we can read off from this that $Dh_{x_0}(e_1^L)=e_1^F$. Also, note that $b$ is uniformly bounded as $\wt{\omega}$ is uniformly $C^1$.

It only remains to check that $c$ is constant. For this, it suffices to show that there exists a fixed $C\neq 0$ such that $Dh_{x_0}(e_2^L)=Ce_2^F+de_1^F$ for some function $d$. We check this using properties of the normal derivative. 
Because $h_{x_0}$ is $C^1$ and intertwines the $\mc{S}^L$ and $\mc{S}^F$ foliations, by Claim \ref{claim:normal_deriv_coincides_with_old_deriv}, we have that at places where $h_{x_0}$ is differentiable the normal derivative agrees with the actual derivative. This means that if $\Pi_{e_2^F}^{e_1^F}$ denotes the projection of a vector onto $e_2^F$ along $e_1^F$, then as $h_{x_0}$ is $C^1$, 
\[
ce_2^F=\Pi_{e_2^F}^{e_1^F}Dh_{x_0}(e_2^L)=(D_{e_2^L,e_2^F}h_{x_0})e_2^F.
\] 
So by Lemma \ref{lem:normal_deriv_constant}, at almost every point $D_{e_2^L,e_2^F}h_{x_0}=C$.  This implies that almost everywhere $c=C$. But $c$ is continuous so $c=C$ everywhere.
\end{proof}

\begin{lem}\label{lem:C_is_1}
In Lemma \ref{lem:smoothing_with_C}, we must have $C=1$.
\end{lem}
\begin{proof}
For the sake of contradiction, suppose that $C\neq 1$.
We write $h_n=F^{-n}\circ h_0\circ L^n$ and apply Lemma \ref{lem:delallave}. The differential of $h_n$ at $z$ is equal to:
\[
\lambda^{-n}\begin{bmatrix}
1 & -\lambda\\
0 & 1
\end{bmatrix}^n
\begin{bmatrix}
1 & b(L^n(z))\\
0 & C
\end{bmatrix}
\lambda^n
\begin{bmatrix}
1 & 1/\lambda\\
0 &1
\end{bmatrix}^n
=
\begin{bmatrix}
1 & \frac{n}{\lambda}-\frac{Cn}{\lambda}+b(L^n(z))\\
0 & C
\end{bmatrix}.
\]
Note that the upper right hand corner of this matrix is not only unbounded, but is going uniformly to infinity because $b(L^n(z))$ uniformly bounded. Let $\gamma\colon [0,1]\to \mc{W}^{u,L}$ be the curve $t\mapsto x+te_2^L$. Then by integrating the derivative of $h_n$, we see that arbitrarily small segments of $\gamma$ are stretched to uniform length by $h_n$ for sufficiently large $n$. But this implies that the sequence $h_n$ can have no uniform modulus of continuity, which contradicts the uniform convergence of $h_n$ to a continuous function $h$.
\end{proof}

We can now show that $h$ is Lipschitz on each leaf.

\begin{proof}[Proof of Lemma \ref{lem:lipschitz_conj}.]We can now conclude by applying Lemma \ref{lem:delallave} again. By applying Lemma \ref{lem:smoothing_with_C} and Lemma \ref{lem:C_is_1}, we obtain a map $h_0$ that is uniformly $C^{1+}$ on $\mc{W}^{u,L}$ leaves and such that the differential $Dh_0$ is, with respect to the frames $[e_1^L,e_2^L]$, $[e_1^F,e_2^F]$,
\[
\begin{bmatrix}
1 & a(z)\\
0 & 1
\end{bmatrix},
\]
where $a$ is uniformly bounded.
As before, let $h_n=F^{-n}\circ h_0\circ L^n$. Then we may calculate as before that the differentials of the functions $h_n$ are
\[
\begin{bmatrix}
1 & a(L^n(x))\\
0 & 1
\end{bmatrix}.
\]
But note that as these differentials are uniformly bounded that the sequence $h_n$ is uniformly Lipschitz. By Lemma \ref{lem:delallave}, as $h$ is the uniform limit of the $h_n$, we obtain that $h$ is uniformly Lipschitz on each leaf.
\end{proof}

\section{Differentiability of the conjugacy.}
We can now prove Proposition \ref{prop:framed_implies_smooth}.

\begin{proof}[Proof of Proposition \ref{prop:framed_implies_smooth}.]
The hypotheses include that $F$ is $U_{\lambda}$-framed. Hence by Lemma \ref{lem:lipschitz_conj}, $h$ is Lipschitz and hence differentiable almost everywhere along unstable leaves. Its derivative, with respect to the framings $[e_1^L,e_2^L]$ and $[e_1^F,e_2^F]$ satisfies
\[
\begin{bmatrix}
\lambda & 1\\
0 & \lambda
\end{bmatrix}D_xh=
D_{L(x)}h\begin{bmatrix}
\lambda & 1\\
0 & \lambda
\end{bmatrix}.
\]
 In particular, by \cite[Thm. 2.1]{kalinin2022local}, $Dh$ agrees almost everywhere with a H\"older continuous function. 
 In particular, this implies that $h$ is uniformly $C^{1+}$ along unstable manifolds as $h$ is Lipschitz and hence is the integral of its derivative. The same argument shows that $h$ is $C^{1+}$ along stable manifolds.  Thus by Journ\'e's lemma $h$ is $C^{1+}$ (see \cite{journe1988regularity} or \cite[Lem. 35]{dewitt2021local}).
\end{proof}

\appendix

\section{Conjugacy of Integer Matrices with Jordan Blocks}

The purpose of this appendix is to give the characterization of the elements of $\SL(4,\Z)$ we consider in this paper.

\begin{proof}[Proof of Proposition \ref{prop:integer_conjugacy_classes}.]
We first show that such a hyperbolic matrix is conjugate to a matrix of the given form.  Let $p(\lambda)$ denote the characteristic polynomial of $L$. Then because $L$'s eigenvalues each have multiplicity $2$ and $L$ is hyperbolic, we see that $p(\lambda)=q(\lambda)^2$ for some irreducible $q\in \Z[\lambda]$. By \cite[Thm. III.12]{newman1972integral}, any integer matrix $L$ is conjugate to a block upper triangular matrix, where the diagonal blocks of the matrix correspond to the irreducible factors of the characteristic polynomial of $L$. In our case, this implies that $L$ is conjugate to a matrix of the given form.
\end{proof}

\begin{rem}
It is possible that $A$ may not be conjugate to $B$ through integer matrices. For example, consider the automorphism defined by 
\begin{equation}
\begin{bmatrix}
3 & 2  & 1 & 0\\
4 & 3 & 0 & 1\\
0 & 0 & 0 & 1\\
0 & 0 & -1 & 6
\end{bmatrix}.
\end{equation}
\end{rem}

\section{Orientability of the unstable bundle}

The purpose of this section is to show that the unstable bundle of an Anosov diffeomorphism on a torus is orientable. This is relatively straightforward to show once one has introduced appropriate definitions.  The central idea is just that topological conjugacies preserve the notion of topological orientability of topological foliations. Hence, if we know a foliation is conjugate to an orientable topological foliation, we will use this information to upgrade that topological orientability to an orientation of the tangent to the foliation.

Let $(\Sigma,\mc{F})$ be a foliation of a $n$-manifold by leaves of dimension $k$ that is given by a foliation atlas $\{(U_i,\phi_i)\}$, where each $\phi_i\colon U_i\subset M\to \R^{k}\times \R^{n-k}$, so that the transition functions have the form
\[
\phi_i\phi_j^{-1}(x,y)=(\alpha_{ij}(x,y)),\gamma_{ij}(y)).
\]

Note that even if the maps $\phi_i$ are $C^0$, it makes sense to say whether the transition function is orientation preserving. This is because for a fixed $y_0$, $\alpha_{ij}(x,y_0)$ as a topological map may preserve or reverse the topological orientation at any given point, i.e. its map on local homology is $\pm 1$ at every point. If this map is $1$ at every point, then $\alpha_{ij}$ is orientation preserving. If $\alpha_{ij}$ is orientation preserving for each point in its domain, then we say that the transition function $\phi_i\phi_j^{-1}$ is orientation preserving

\begin{defn}
\cite[Sec. 2.3]{hector1981introduction} We say that a foliation $(\Sigma,\mc{F})$ is orientable if it admits an atlas such that all transition functions $\phi_i\phi_j^{-1}$, as above, are orientation preserving.
\end{defn}

When we speak of an orientation of a foliation $\mc{F}$ with uniformly $C^1$ leaves, what we mean is that there exists an orientation on $T\mc{F}$ as a continuous bundle. The following proposition says that if a foliation with uniformly $C^1$ leave is orientable as a topological foliation, then it is orientable.

\begin{prop}\label{prop:orientable_foliations_lemma}
Suppose that $M$ is a $C^1$ manifold and that $\mc{F}$ is a foliation of $M$ with uniformly $C^1$ $k$-dimensional leaves. Then if $\mc{F}$ admits an oriented topological foliation atlas, then $T\mc{F}$ is orientable as a continuous bundle over $M$.
\end{prop}

\begin{proof}
We will construct a continuous non-vanishing section of $\Omega^k T\mc{F}$, the space of $k$-forms on $T\mc{F}$. Fix a continuous Riemannian metric on $TM$. 

If we have a point $x\in M$, then $x$ is in the domain of some foliation chart $\phi\colon U\to \R^k\times \R^{n-k}$. 
If we let $U_x$ be the plaque of this chart containing $x$, then we may find a $C^1$ disk $D_x$ containing $x$ inside of $U_x$ whose orientation agrees with the chart orientation on the plaque. 
As $\mc{F}$ is uniformly $C^1$, we may fix a local smooth transversal foliation $\mc{T}$ defined in a neighborhood of $U_x$. 
This gives us holonomy maps $T_{x,y}\colon D\to D_y$, between $D$ and its image in the plaque $U_y$ via the $\mc{T}$ holonomies. 
Let $\omega\in \Omega^k TD$ be a choice of orientation on $D_x$ agreeing with the chart orientation. 
Then $D(T_{x,y})_*\omega$ gives an orientation on $D_y$. 
Thus on a neighborhood $V\subset M$ of $x$, we may define a continuous unit norm section $\hat{\omega}$ of $\Omega^k T\mc{F}$ by setting $\hat{\omega}=\omega/\|\omega\|$. 

Note that the oriented atlas of $\mc{F}$ gives at every point $z\in M$ a choice of generator of the local homology $[\mu_z]\in H_n(\mc{F}(z),\mc{F}(z)\setminus z)$. But for such a  $C^1$ manifold endowed with a metric, $[\mu_z]$ is associated to a unique element $\hat{\omega}_x\in \Omega^{\dim M}_z TM$ of unit norm. If we let $\{\mu_z\}_{z\in M}$ be the choice of generators of $H_n(\mc{F},\mc{F}(z)\setminus z)$ arising from the topological oriented foliation atlas, then observe that $[\mu_z]$ and $\hat{\omega}$ correspond to the same orientation. Thus we may extend the definition $\hat{\omega}$ globally because any section we construct agrees with $[\mu_z]$. $\hat{\omega}$ is then continuous due to its local construction.
\end{proof}

For a discussion of Anosov automorphims of nilmanifolds, see \cite[Sec. 2]{dewitt2021local}. Note that a torus is a nilmanifold, so the following result applies in our setting.
\begin{prop}\label{prop:Anosov_unstable_is_orientable}
Suppose that $F\colon N/\Gamma\to N/\Gamma$ is an Anosov diffeomorphism of a nilmanifold. Then the unstable bundle $E^{u,F}$ is orientable.
\end{prop}

\begin{proof}
By the work of Franks
\cite{franks1969anosov}
 and Manning \cite{manning1974there}
, there exists an Anosov automorphism $L\colon N/\Gamma\to N/\Gamma$ and a conjugacy $h$ between $F$ and $L$. From the discussion in \cite[Sec. 2.2]{dewitt2021local}, we see that the foliation $\mc{W}^{u,L}$ is orientable. Thus as $h$ intertwines the foliation $\mc{W}^{u,F}$ and $\mc{W}^{u,L}$, by Proposition \ref{prop:orientable_foliations_lemma}, $T\mc{W}^{u,F}=E^{u,F}$ is an orientable bundle.
\end{proof}

\bibliographystyle{amsalpha}
\bibliography{biblio.bib}

\end{document}